\documentclass[mn,fleqn]{w-art}
\usepackage{amssymb}
\usepackage{times,cite,w-thm}
 \numberwithin{equation}{section}

\theoremstyle{plain}

\theoremstyle{definition}

\renewcommand\proofname{Proof}
\def\C{{\mathbb{C}}}
\def\N{{\mathbb{N}}}
\def\R{{\mathbb{R}}}
\newcommand{\gG}{{\Gamma}}
\newcommand{\kH}{{\mathcal H}}
\newcommand{\kC}{{\mathcal C}}

\newcommand{\bR}{{\mathbb R}}
\newcommand{\bC}{{\mathbb C}}
\def\cH{{\mathcal H}}
\newcommand\dom{\operatorname{dom}}
\newcommand{\gotH}{{\mathfrak H}}
\def\Ext{{\rm Ext\,}}
\newcommand{\supp}{\mathop{\rm supp}\nolimits}
\newcommand{\ran}{{\mathrm{ran\,}}}
\newcommand{\diag}{\mathop{\rm diag}\nolimits}

\DeclareMathOperator{\Span}{span}\DeclareMathOperator{\imm}{Im}\DeclareMathOperator{\sh}{sh}

\DeclareMathOperator{\rank}{rank} \DeclareMathOperator{\cl}{cl}
\DeclareMathOperator{\loc}{loc}
\def\mul{{\rm mul\,}}\def\op{{\rm op\,}}

\begin{document}

\DOIsuffix{theDOIsuffix}
\Volume{}
\Month{}
\Year{}
\pagespan{1}{}
\Receiveddate{}
\Reviseddate{}
\Accepteddate{}
\Dateposted{}

\title[Spectral theory  of the  Schr\"odinger operators with point
interactions]{Radial  positive   definite functions and
 spectral theory  of the  Schr\"odinger operators with point
interactions}

\author[N. Goloshchapova]{N. Goloshchapova\inst{1,}%
  \footnote{Corresponding author\quad E-mail:~\textsf{nataliia@ime.usp.br},
           }}
\address[\inst{1}]{R. Luxemburg str. 74, \, 83114 Donetsk, Ukraine}
\author[M. Malamud]{M. Malamud \inst{1,}\footnote{E-mail:~\textsf{mmm@telenet.dn.ua}.}}

\author[V.  Zastavnyi]{ V. Zastavnyi \inst{2,}\footnote{E-mail:~\textsf{zastavn@rambler.ru}.}}
\address[\inst{2}]{Univesitetskaja str. 24,\, 83001 Donetsk, Ukraine}
    \dedicatory{ Dedicated   to  the  75th  anniversary of  Eduard  Tsekanovskii.}
\begin{abstract}
  We  complete the  classical    Schoenberg representation  theorem for  radial
positive definite  functions. We  apply this result  to study
spectral properties  of self-adjoint realizations   of two- and
three-dimensional Schr\"odinger operators with
point interactions  on a  finite  set. In particular,  we prove  that  any realization has  purely absolutely continuous non-negative spectrum.
\end{abstract}

\maketitle

\section{Introduction}
An important topic in quantum mechanics is the spectral theory of Schr\"odinger operators on the Hilbert space $L^2(\mathbb{R}^d),\,d\in\{1,2,3\},$ with
potentials supported on a discrete (finite or countable) set of points of $\mathbb{R}^d$.
 There is an extensive
literature on such operators (see \cite{Ada07,AGHH88,AK1,ArlTse05,BF,bmn,hk,LyaMaj,ogu10,Pos08} and
references therein).
The first  mathematical problem is to associate a self-adjoint operator
(Hamiltonian) on $L^2(\mathbb{R}^d)$ with the differential expression
\begin{equation}\label{eq0}
 \mathfrak{L}_d:= -\Delta+\sum\limits_{j=1}^m\alpha_j\delta(\cdot - x_j),\quad \alpha_j\in \mathbb{R},\,\,\,m \in\mathbb{N}.
 \end{equation}
 There are at least two    natural ways to associate  self-adjoint operator $H_{X,\alpha}$ with the   following differential expression  in $L^2(\mathbb{R}^1)$
\begin{equation*}
\mathfrak{L}_1:= -\frac{\textrm{d}^2}{\textrm{dx}^2}+\sum\limits_{j=1}^m\alpha_j\delta(\cdot - x_j),\,\,\,m \in\mathbb{N}
\end{equation*}
 for any   fixed set $\alpha:=\{\alpha_j\}_{j=1}^m\subset\mathbb{R}$.
The first  one  is  based on  the  quadratic  forms  method.  Another way to  introduce   local interactions  on  $X:=\{x_j\}_{j=1}^m\subset\mathbb{R}$   is to  consider  maximal  operator corresponding to  $\mathfrak{L}_1$  and impose boundary  conditions at $x_j, \quad j\in\{1,..,m\}$ (see \cite{AGHH88}), i.e.,
\[\dom(H_{X,\alpha}) = \{f\in W^{2,2}(\mathbb{R}\setminus X) \cap W^{1,2}(\mathbb{R}) : f'(x_j+) - f'(x_j-) = \alpha_jf(x_j)\}.\]

In contrast  to  one-dimensional   case,   the  differential  expression \eqref{eq0}
   does  not define an operator  in $L^2(\mathbb{R}^d),\\d\geq 2$,  by  means of  the quadratic   forms since  the linear functional $\delta_x :\,f\rightarrow f(x)$  is not  continuous in $W^{1,2}(\mathbb{R}^d)$  for $d\geq 2$.
 However, it is still possible to apply the extension theory approach.
  Namely, F. Berezin and
L.~Faddeev in their pioneering paper  \cite{BF}  proposed to consider \eqref{eq0} (with $m = 1$ and $d = 3$) in the
framework of extension theory. They     associated   with $\mathfrak{L}_{d}$ the family  of all self-adjoint  extensions  of the  following  symmetric  operator
    \begin{equation}\label{min}
H:=-\Delta,\quad \dom(H) :=\bigl\{f \in W^{2,2}(\bR^d):
f(x_j)=0,\quad j\in \{1,..,m\} \bigr\},\,\, m\in\N.
\end{equation}
It is well known that    $H$ is  closed non-negative symmetric
operator with equal
 deficiency  indices $n_{\pm}(H)=m$ (see \cite{AGHH88}).
   In  \cite{AGHH88}, the  authors proposed to associate   with the Hamiltonian \eqref{eq0}  certain    $m$-parametric  family  $H_{X,\alpha}^{(d)}$   describing  local   point  interactions. They    parameterized the   family      in terms of the  resolvents.
   The latter  enabled   the  authors   to
 obtain  an explicit description of  the   spectrum for any operator from the  family  $H_{X, \alpha}^{(d)}$.

  In  the recent  publications \cite{ArlTse05,bmn,hk},  boundary  triplets and the corresponding Weyl  functions
  technique
   (see   \cite{DerMal91, GG} and also  Section \ref{prelim}) was involved
  to investigate  multi-dimensional Schr\"odinger operators with point
  interactions (the  cases $d\in\{2,3\}$).
 In  the present paper, we  apply  boundary  triplets  approach   to  parametrize    all  self-adjoint   extensions  of  $H$. 
  Besides, using Weyl functions technique, we investigate their    spectra.  
   Moreover, we   substantially involve  the theory  of  radial  positive definite functions   \cite[chapter V]{Akh65}  in our approach.
  In particular,  we employ strict  positive  definiteness   (see  Definition \ref{defpoz}) of the
  functions $\tfrac{\sin s|\cdot|}{s|\cdot|}$   and the  Bessel   functions  $J_0(s|\cdot|)$  with  any $s>0$ on
  $\mathbb{R}^3$ and  $\mathbb{R}^2$, respectively, to  prove  pure absolute  continuity of  non-negative spectrum of any self-adjoint
  realization  of $\mathfrak{L}_d$.  For this purpose we  complete   the classical  Schoenberg
  theorem \cite{Sch38}  regarding  the integral  representation  of radial positive   definite  functions.

The paper  is organized as follows. Section 2  is  introductory.
It contains definitions of   a  boundary  triplet and the  corresponding  Weyl  function \cite{DerMal91,GG}  and  also  facts about the   Weyl  functions  \cite{BraMal02,MalNei11}. In Section 3,  we complete    Schoenberg
  theorem  by establishing  strict positive definiteness  of  any  non-constant
  radial  positive definite   function on  $\mathbb{R}^n,\, n\geq 2$.
  In   Sections 4 and
 5,   we   investigate   3D- and 2D-Schr\"odinger
  operators with  point  interactions,
  respectively.

    Namely, in Subsection 4.1 (resp., 5.1),  we  define  boundary triplet $\Pi=\{\mathcal{H},\Gamma_0,\Gamma_1\}$
    for   $H^*$ and  compute the  corresponding
    Weyl  function.  It appears  to  be close  to  that contained  in  Krein's  resolvent  formula for the  family  $H_{X,\alpha}^{(d)}$  in \cite{AGHH88}. In particular,  for the proof of  the  surjectivity  of the  mapping
  $\Gamma=(\Gamma_0,\Gamma_1)^\top$  we employ   the   strict  positive  definiteness of
   the function $e^{-|\cdot|}$ on  $\mathbb{R}^n$  for any  $\,n\in\mathbb{N}$.

Subsections 4.2 and 5.2 are  devoted  to the  spectral
 analysis of the   self-adjoint  realizations of $\mathfrak{L}_d$.  To  investigate   the  absolutely continuous
 spectrum
we apply technique elaborated in \cite{BraMal02, MalNei11}. For
this purpose we  need invertibility of the matrices
    \[
\left(\delta_{kj}+\frac{\sin(\sqrt{x}|x_k-x_j|)}{\sqrt{x}|x_k-x_j|+\delta_{kj}}\right)_{j,k=1}^m
\quad \text {and}\quad
\left(J_0(\sqrt{x}|x_j-x_{k}|)\right)_{j,k=1}^m,\quad x\in \R_+,
     \]
which we extract from the  strict positive definiteness of  the
functions $\tfrac{\sin s|\cdot|}{s|\cdot|}$ and $J_0(s|\cdot|),\,s>0,$ on  $\R^3$ and
$\R^2$, respectively. We emphasize  that in the   proof of Theorems \ref{spec3}  and \ref{spec2} our complement to Schoenberg  theorem is  used in full generality. Indeed, it follows  from the  integral representation \eqref{schonberg} that the  strict  positive  definiteness of $\tfrac{\sin s|\cdot|}{s|\cdot|}$ and $J_0(s|\cdot|)$  for all  $s>0$ yields the strict positive  definiteness   of  any   radial positive  definite function $f$    on
  $\mathbb{R}^3$ and  $\mathbb{R}^2$, respectively.

     Finally,  description  of
  the non-negative self-adjoint  extensions  of  $H$ ($d=3$) is  provided   in   Subsection
  4.3.  For   suitable choice of a   boundary  triplet  $\widetilde{\Pi}$
  strong resolvent limit of the  corresponding   Weyl  function
  $\widetilde{M}(x)$   at  $x=0$ appears to be positive  definite  matrix
  in a view of   strict  positive   definiteness   of  the function
  $\tfrac{1-e^{-|\cdot|}}{|\cdot|}$ \,\, on\,\,  $\mathbb{R}^n$    for  any $\,n\in\mathbb{N}$. 

\textbf{Notation.} Let   $\gotH$ and $\kH$ denote separable
Hilbert spaces; $[{\gotH}, {\kH}]$ denotes  the space of bounded
linear operators from ${\gotH}$ to ${\kH}$, $[\kH]:=[\kH,\kH]$;
the set of closed operators in $\kH$ is denoted by $\kC(\kH)$. Let
$A$ be a linear operator in a Hilbert space $\mathfrak{H}$. In  what follows domain, kernel, and range of $A$ are  denoted   by
 $\dom (A)$, $\ker (A)$, $\ran (A)$, respectively; $\sigma(A)$ and $\rho (A)$
denote the spectrum and the resolvent set of $A$; $\mathfrak{N}_z$
denotes the defect subspace of $A$; 
\,$C[0,\infty)$ denotes   the Banach space of functions
continuous and  bounded on $[0,\infty)$.
\section{Extension  theory  of  symmetric  operators}
\subsection{Boundary  triplets and proper extensions}\label{prelim}

 In this
subsection, we recall basic notions and facts of the theory of
boundary triplets (we refer the reader to \cite{DerMal91, DerMal95, GG} for
a detailed exposition). In what follows  $A$  always denotes a closed symmetric operator in  separable Hilbert space
$\gotH$ with equal deficiency indices $n_-(A)=n_+(A)\leq \infty$.
    \begin{definition}\cite{GG}\label{bound}%
\,\,A totality $\Pi=\{\kH,\gG_0,\gG_1\}$ is called a {\rm boundary
triplet} for the adjoint operator $A^*$ of $A$ if $\kH$ is an
auxiliary Hilbert space and
$\Gamma_0,\Gamma_1:\  \dom(A^*)\rightarrow\kH$ are linear mappings such that\\
 $(i)$ the following abstract second  Green identity holds
\begin{equation}\label{GI}
(A^*f,g)_\gotH - (f,A^*g)_\gotH = (\gG_1f,\gG_0g)_\kH -
(\gG_0f,\gG_1g)_\kH,\qquad f,g\in\dom(A^*);
\end{equation}
%
%
%
$(ii)$ the mapping $\gG:=(\Gamma_0,\Gamma_1)^\top: \dom(A^*)
\rightarrow \kH \oplus\kH$ is surjective.
   \end{definition}
   With   a  boundary   triplet   $\Pi=\{\kH,\gG_0,\gG_1\}$ for $A^*$  one  associates two    self-adjoint extensions of $A$  defined  by
\[
 A_0:=A^*\!\upharpoonright\ker(\gG_0)\quad  \text{and}\quad
A_1:=A^*\!\upharpoonright\ker(\gG_1).
\]

    \begin{definition}
$(i)$ A closed extension $\widetilde{A}$ of $A$ is called
\emph{proper} if $A\subseteq\widetilde{A}\subseteq A^*$.  The set
of all  proper  extensions of  $A$  is  denoted  by $\Ext_A$.

$(ii)$ Two proper extensions $\widetilde{A}_1$ and $\widetilde{A}_2$
of $A$ are called \emph{disjoint}
 if $\dom(\widetilde{A}_1)\cap\dom(\widetilde{A}_2)=\dom(A)$.
           \end{definition}
%
%

       \begin{remark}
$(i)$  For any symmetric operator $A$ with  $n_+(A)=n_-(A)$, a
boundary triplet $\Pi=\{\kH,\gG_0,\gG_1\}$ for $A^*$ exists and is
not unique \cite{GG}. It is  known also that  $\dim \kH =
n_{\pm}(A)$ and $\ker\gG =\ker(\Gamma_0,\Gamma_1)^\top = \dom(A).$

 $(ii)$ Moreover, for each self-adjoint extension
 $\widetilde{A}$ of $A$  there  exists  a  boundary   triplet
 $\Pi=\{\mathcal{H},\Gamma_0,\Gamma_1\}$  such  that
 $\widetilde{A}=A^*\upharpoonright\ker(\Gamma_0) =: A_0.$


 $(iii)$ For   each  boundary  triplet $\Pi=\{\mathcal{H},\Gamma_0,\Gamma_1\}$ for  $A^*$  and each
 bounded self-adjoint  operator $B$  in  $\mathcal{H}$    a triplet
 $\Pi_B=\{\mathcal{H},\Gamma_0^B,\Gamma_1^B\}$   with   $\Gamma_1^B:=\Gamma_0$  and
 $\Gamma_0^B:=B\Gamma_0-\Gamma_1$ is also  a  boundary  triplet   for  $A^*$.
 \end{remark}

 A   role  of  a  boundary triplet for $A^*$  in   the   extension theory is   similar
 to that of coordinate system in the analytic  geometry. Namely,  it allows  one  to  parameterize
 the  set  $\Ext_A$   by  means of  linear  relations in $\kH$ in place of J.von Neumann formulas.
To  explain this we
recall the following   definitions.
\begin{definition}
$(i)$  A  closed  linear relation $\Theta$ in  $\mathcal{H}$ is
a closed subspace of $\mathcal{H}\oplus\mathcal{H}$.

$(ii)$  A linear  relation  $\Theta$  is symmetric
 if  $(g_1,f_2)-(f_1,g_2)=0$  for all  $\{f_1,g_1\},  \{f_2,g_2\}\in  \Theta$.

 $(iii)$ The adjoint relation
$\Theta^*$  is defined by
\begin{equation*}
\Theta^*= \left\{ \{k,k'\}
: (h^\prime,k)=(h,k^\prime)\,\,\text{for all}\,
\{h,h'\}
\in\Theta\right\}.
\end{equation*}

 $(iv)$  A closed linear   relation $\Theta$  is  called  self-adjoint
 if both $\Theta$ and $\Theta^*$ are  maximal symmetric, i.e., they do not admit symmetric extensions.
    \end{definition}
     For the
symmetric relation $\Theta\subseteq\Theta^*$ in $\cH$ the
multivalued part $\mul(\Theta)$ is the orthogonal complement of
$\dom(\Theta)$ in $\cH$. Setting $\cH_{\rm
op}:=\overline{\dom(\Theta)}$ and $\cH_\infty=\mul(\Theta)$, one
verifies that $\Theta$ can be rewritten as the direct orthogonal sum
of a self-adjoint operator $\Theta_{\rm op}$ in the subspace
$\cH_{\rm op}$ and a ``pure'' relation
$\Theta_\infty=\bigl\{\{0,f'\}:f'\in\mul(\Theta)\bigr\}$ in the
subspace $\cH_\infty$.

%
%



\begin{proposition}\cite{DerMal91,GG}\label{propo}
Let  $\Pi=\{\mathcal{H},\Gamma_0,\Gamma_1\}$  be   a boundary
triplet for $A^*$. Then  the  mapping
\begin{equation}\label{s-aext}
\Ext_A\ni\widetilde{A}:=A_\Theta\rightarrow \Theta:=\Gamma(\dom(\widetilde{A}))=\{\{\Gamma_0f,\Gamma_1f\}:\,\,f\in \dom(\widetilde{A})\}
\end{equation}
establishes  a  bijective  correspondence   between  the  set  of
all  closed proper extensions $\Ext_A$  of $A$ and  the set of all
closed   linear  relations $\widetilde{\mathcal{C}}(\mathcal{H})$
in $\mathcal{H}$. Furthermore,  the following  assertions  hold.

$(i)$
  The  equality $(A_\Theta)^*=A_{\Theta^*}$
%
%
holds  for any  $\Theta\in\widetilde{\mathcal{C}}(\mathcal{H})$.

$(ii)$
 The extension $A_\Theta$  in \eqref{s-aext} is  symmetric
(self-adjoint)  if  and only  if \,  $\Theta$  is symmetric
(self-adjoint). Moreover, $n_{\pm}(A_\Theta) = n_{\pm}(\Theta).$

$(iii)$
If,  in  addition,   the closed extensions $A_\Theta$ and
$A_0$ are disjoint, then   \eqref{s-aext} takes the form
\begin{equation*}\label{bijop}
A_\Theta=A_B=A^*\!\upharpoonright\dom(A_B) ,\quad \dom (A_B)=\dom
(A^*)\!\upharpoonright\ker\bigl(\Gamma_1-B\Gamma_0\bigr),\quad  B\in\mathcal{C}(\mathcal{H}).
   \end{equation*}
%
%
 \end{proposition}
%
%


\subsection{Weyl   function, $\gamma$-field  and spectra of proper extensions}

It  is  known that  Weyl  function is an important tool in
the spectral  theory of singular  Sturm-Liouville  operators.  In
\cite{DerMal91,DerMal95} the concept of Weyl function was
generalized to   an  arbitrary symmetric operator $A$  with
equal deficiency indices.  In  this  subsection we
recall basic  facts about  Weyl   functions.
\begin{definition}\cite{DerMal91}\label{Weylfunc}
Let $\Pi=\{\kH,\gG_0,\gG_1\}$ be a boundary triplet  for $A^*.$
The operator valued function   $M(\cdot) :\
\rho(A_0)\rightarrow  [\kH]$ defined by
\begin{equation}\label{2.3A}
 M(z)\Gamma_0f_z=\Gamma_1f_z, \quad f_z\in \mathfrak{N}_z,\,\,z\in\rho(A_0),
      \end{equation}
is  called  the {\em Weyl function}, corresponding to the triplet $\Pi.$
      \end{definition}
%
The  definition  of the  Weyl function  is  correct   and the Weyl function $M(\cdot)$  is Nevanlinna  or  R-function.

 In  the   following  we  will be concerned  with a simple symmetric  operators.  Recall  that a  symmetric  operator $A$ is  said  to  be  \emph{simple} if there is no nontrivial subspace  which reduces  it to self-adjoint  operator.

 The spectrum and the resolvent set of the closed (not
necessarily self-adjoint) extensions of  simple symmetric  operator $A$ can be described with
the help of the function $M(\cdot)$.  Namely, the  following  proposition  holds.
  \begin{proposition}\label{prop_II.1.4_spectrum}
Let $A$ be a  densely defined  simple symmetric operator in $\gotH$,
$\Theta\in \widetilde{\mathcal{C}}(\mathcal{H})$, $A_\Theta\in \Ext_A$,
and $z\in \rho(A_0)$. Then the following equivalences hold.

(i)\  $z\in \rho(A_\Theta) \quad
\Longleftrightarrow\quad 0\in \rho(\Theta-M(z));$

(ii) \ $z\in\sigma_\tau(A_\Theta) \quad
\Longleftrightarrow\quad 0\in \sigma_\tau(\Theta-M(z)),\qquad
\tau\in\{ p,\ c,\ r\};$

(iii)\
$f_z\in \ker(A_\Theta - z) \quad \Longleftrightarrow\quad
\Gamma_0f_z\in \ker(\Theta - M(z))$\quad
and\ \  $\dim\ker(A_\Theta -z) = \dim\ker(\Theta - M(z)).$
      \end{proposition}

The  following
proposition   gives us quantitative  characterization  of the   negative spectrum of  self-adjoint  extensions of the  operator $A$.
   \begin{proposition}\cite{DerMal91}\label{prkf}
Let $A$ be  a densely defined non-negative   symmetric operator  in
$\gotH$, and  let   $\Pi=\{\kH,\gG_0,\gG_1\}$ be a  boundary
 triplet for  $A^*$. Let  also  $M(\cdot)$ be the corresponding
Weyl function  and $A_0=A_F$,  where  $A_F$  is the  Friedrichs   extension  of $A$. Then the   following  assertions hold.

$(i)$ The  strong  resolvent  limit   $M(0):=s-R-\lim\limits_{x\uparrow 0} M(x)$   exists and  is   self-adjoint  linear  relation  semi-bounded  from   below.

$(ii)$   If,  in  addition,  $M(0)\in[\cH]$,   then   the number  of negative squares  of    $A_\Theta=A_\Theta^*$
equals the number of negative squares of the  relation
$\Theta-M(0)$, i.e.,
$\kappa_{-}(A_\Theta)=\kappa_{-}(\Theta-M(0))$.

In  particular,  the self-adjoint extension $A_\Theta$ of $A$ is  non-negative
if and only  if the linear relation $\Theta-M(0)$ is non-negative.
\end{proposition}
Denote
  \begin{gather*}
M(x + i0):=s-\lim_{y\downarrow0}M(x+iy),\quad
d_M(x):=\rank(\imm(M(x+i0))),\\
M_h(z):=(M(z)h,h),\,\,\Omega_{ac}(M_h):=\{x\in\mathbb{R}:\,0<\imm (M_h(x+i0))<+\infty\},\,\, z\in\mathbb{C_+},\,\, h\in\mathcal{H},
  \end{gather*}
where  $M_h(x+i0):=\lim_{y\downarrow0}(M(x+iy)h,h)$. Since  $\imm(M_h(z))>0,\ z\in \mathbb{C}_+,$ the limit  $M_h(x+i0)$
exists and is finite for  a.e. $x\in\mathbb{R}.$
%
%

%
%
%

%
%

  %
  %
  %
  %
To state the next proposition we need a concept of the absolutely
continuous closure $\cl_{ac}(\delta)$ of a Borel subset
$\delta\subset \R$ introduced in \cite{BraMal02} and \cite{Ges08}.
We refer to \cite{Ges08,MalNei11} for the definition  and basic properties.
       \begin{proposition}\cite{BraMal02,MalNei11}\label{ac}
Let $A$  be a simple densely defined closed  symmetric operator  with equal deficiency indices in  separable Hilbert space  $\gotH$. Let $\Pi=\{\kH,\gG_0,\gG_1\}$ be a boundary triplet  for
$A^*$ and  $M(\cdot)$   the corresponding  Weyl function.
Assume also that  $\tau=\{h_k\}_{k=1}^N,\quad 1\leq N\leq\infty$  is a
total set in $\mathcal{H}$.\,
 Let  also
$A_B=A^*\upharpoonright\ker(\Gamma_1-B\Gamma_0)$, with
$B=B^*\in\mathcal{C}(\mathcal{H})$.   Then   the  following
assertions  hold.

$(i)$ The operator $A_0$ has  no singular  continuous spectrum
within the   interval  $(a,b)$ if   for  each
$k\in\{1,2,..,N\}$ the set $(a,b)\setminus\Omega_{ac}(M_{h_k})$ is
 countable.

 $(ii)$
 If the
limit $M(x+i0)$ exists for a.e.
$x\in\mathbb{R}$, then $\sigma_{ac}(A_0)=\cl_{ac}(\supp(d_M(x)))$.

$(iii)$
For  any  Borel subset   $\mathcal{D}\subset \mathbb{R}$ the  absolutely  continuous  parts
$A_0E^{ac}_{A_0}(\mathcal{D})$ and $A_BE^{ac}_{A_B}(\mathcal{D})$  of the operators $A_0E_{A_0}(\mathcal{D})$  and  $A_BE_{A_B}(\mathcal{D})$
are  unitarily equivalent if and only if\,\, $d_M(x)=d_{M_B}(x)$ for
a.e. $x\in\mathcal{D}.$
\end{proposition}

\section{Positive definite  functions. Complement   of the  Schoenberg theorem}
 Let  $(u,v)=u_1v_1+\ldots+u_nv_n$ be  a scalar
product  of  two vectors  $u=(u_1,\ldots,u_n)$ and
$v=(v_1,\ldots,v_n)$ from $\R^n$, $n\in\N$,   and  let
$|u|=\sqrt{(u,u)}$ be Euclidean norm.  Recall  some  basic facts
and  notions  of the theory of positive  definite  functions
\cite{Akh65}.
\begin{definition}\label{defpoz}\cite{Akh65,Wend_2005}
$(i)$  Function $g(\cdot):\R^n\to\C$ is said to be positive definite on
$\R^n$ and is referred to the class $\Phi(\R^n)$ if it is
continuous at $0$ and for any finite subsets
$X:=\{x_k\}_{k=1}^{m}\subset\R^n$ and
$\xi:=\{\xi_k\}_{k=1}^{m}\subset\C,\,\,m\in\mathbb{N}$ the following inequality holds
\begin{equation}\label{positiv}
\sum_{k,j=1}^{m}\xi_k\overline{\xi}_jg(x_k-x_j)\ge 0.\,
\end{equation}
$(ii)$  Moreover,  $g(\cdot)$ is said to  be  strictly  positive on
$\R^n$ if  the inequality ~\eqref{positiv} is strict for any
subset of  distinct points  $X=\{x_k\}_{k=1}^{m}\subset\R^n$
\,and for  any  subset $\xi=\{\xi_k\}_{k=1}^{m}\subset\C$  satisfying condition
$\sum_{k=1}^m|\xi_k|>0$.
  \end{definition}
Clearly,  positive definiteness of the function  $g(\cdot)$ is
equivalent to the non-negative  definiteness  of  the  matrix
$G(X)=\left(g_{kj}\right)_{k,j=1}^m$ with $g_{kj}=g(x_k-x_j)$ for
any subset $X=\{x_k\}_{k=1}^{m}\subset\R^n$, while its strict
positive definiteness  is equivalent  to (strict)
positive definiteness  of the matrix $G(X)$ for any
subset   of  distinct  points $X=\{x_k\}_{k=1}^{m}\subset\R^n$.

The  following  classical  Bochner theorem gives the   description of   the  class $\Phi(\mathbb{R}^n)$.
 \begin{theorem}\cite{Boch}\label{boch}
A  function $g(\cdot)$  is  positive  definite  on  $\R^n$ if  and
only if
\begin{equation*}
g(x)=\int\limits_{\R^n}e^{i(u,x)}d\mu(u),
\end{equation*}
where  $\mu$ is a  finite non-negative Borel  measure on  $\R^n$.
    \end{theorem}
\begin{definition}
  A function $f(\cdot)\in
C[0,+\infty)$    is said to  be radial   positive definite
function of  the class   $\Phi_n$, $n\in\N,$  if
$f(|\cdot|)$ is positive definite  on  $\R^n$, i.e., if
$f(|\cdot|)\in\Phi(\mathbb{R}^n)$ .
\end{definition}
%
%
 %
A characterization of the class $\Phi_n$ is given by the following Schoenberg theorem
 \cite{Sch38_1,Sch38} (see  also \cite[Theorem~5.4.2]{Akh65}).

 \begin{theorem}\label{sch}
 Function  $f(\cdot)$   belongs to  the   class $\Phi_n$  if and only if
\begin{equation}\label{schonberg}
f(t)=\int_{0}^{+\infty}\Omega_n(st)\,d\mu(s)\,,\quad t\ge 0\,,
\end{equation}
 where  $\mu$  is   a non-negative finite Borel measure on $[0,\infty)$,  and
\begin{equation}\label{kernel}
  \Omega_n(t)=\Gamma\left(\frac{n}{2}\right)\,
  \left(\frac{2}{t}\right)^{\frac{n-2}{2}}J_{\frac{n-2}{2}}(t)=
  \sum_{p=0}^{\infty}\left(-\frac{t^2}{4}\right)^p\frac{\Gamma\left(\frac{n}{2}\right)}{p!\,\Gamma\left(\frac{n}{2}+p\right)}\,,
 \end{equation}
\begin{equation}\label{Omaga_n}
\Omega_n(|x|)=\int_{S_{n}}e^{i(u,x)}d\nu_n(u)\,,\,x\in\R^n\,.
\end{equation}
 Here $\nu_n$ is the   Borel measure uniformly  distributed over  the unit  sphere
  $S_{n}$   centered   at  the  origin  and  $\nu_n(S_{n})=1$.
  \end{theorem}
  \begin{remark}
  It is not difficult to  show  that for any  $n\in\mathbb{N}$ the strict inclusion
   $\Phi_{n+1}\subset\Phi_n$  takes  place.
  Indeed, it  is  known  ~\cite{Go}, \cite[Theorem~5]{Zastavnyi_1992}, \cite[Example~1]{Zastavnyi_2000}
  that $(1-|t|)_+^{\delta}\in\Phi_{n}$ if  and only if $\delta\ge\frac{n+1}{2}$.
Earlier Askey  ~\cite{Askey_1973} and Trigub \cite{Trigub_1989}
considered  the case of natural  $\delta$ and proved the necessity
for odd $n$.

On the other hand, the strict  inclusion
   $\Phi_{n+1}\subset\Phi_n$   follows  from another  example. Namely,
  $\mathop{\rm Re}\left( e^{-zt}\right)\in\Phi_n$ if and only if $|\arg z|\le{\pi}/{(2n)}, \,\,\, z\in\C$ (see ~\cite[Theorem~3]{Zastavnyi_1992}).
    Therefore  $\Omega_n\in\Phi_n$, but
   $\Omega_n\not\in\Phi_{n+1}$, $n\in\mathbb{N}$. 
  \end{remark}
Our    complement to  the Schoenberg theorem   reads  as  follows.
  \begin{theorem}\label{thzast}
 If   $f(\cdot)\in\Phi_n$, $n\ge 2$, and
$f(\cdot)\not\equiv const$ on  $[0,+\infty)$, then  the  function $f(|\cdot|)$
is strictly  positive  definite on  $\R^n$.
\end{theorem}
We start with  the  following  auxiliary  lemma   and   present two    different proofs.
\begin{lemma}\label{lezast}
Let $S^r_n(y)$ be  a sphere of  radius $r$  in  $\R^n$    centered  at  $y$ and $n\ge 2$,  let also $X=\{x_k\}_{k=1}^{m}$ be a
subset of $\R^n$ such that $x_p\ne x_j$ as $p\ne j$  and
$\xi=\{\xi_k\}_{k=1}^{m}\subset\C$.  If
\begin{equation}\label{lindep}
\sum_{p=1}^{m}\xi_p e^{i(u,x_p)}=0,\quad \text{for  all}\qquad u\in S^r_n(y),
\end{equation}
 then  $\xi_p=0$ for  $p\in\{1,..,m\}$.
\end{lemma}
\begin{proof}[\textbf{The first  proof}]
Without loss of generality we  may assume that $y=0$ and  $r=1$.  Let  also   for  definiteness $\xi_1\neq 0$.
  For $m=1$
the  statement is obvious. Put $m\geq 2$.
Denote
$\max\limits_{1\leq p\leq m}|x_p|=R_0>0$.  Let $\{e_j\}_{j=1}^n$
 be  the   standard orthogonal  basis  in $\R^n$. We may assume  that  $x_1=R_0e_1$.
 Let $P$  be the orthogonal  projector onto $\Span\{e_1,e_2\}$.  Then $Px_p=r_p(\cos\varphi_pe_1+\sin\varphi_pe_2)$   with  $0\leq r_p=|Px_p|\leq|x_p|\leq R_0, \quad 0\leq\varphi_p<2\pi.$
 Assume that
$r_1=r_2=...=r_{m'}=R_0$   and $r_p<R_0$ as $p>m'$ (if $m'<m$).
Then $Px_p=x_p, \,\,\, 1\leq p\leq m'$. Evidently, we   may   also
assume  that $0=\varphi_1<\varphi_2<\ldots<\varphi_{m'}<2\pi$. Put
in \eqref{lindep} $u= \cos\varphi e_1 + \sin\varphi e_2\in
S^1_n(0), \,\,\varphi\in \R.$
   Therefore the
equality  \eqref{lindep}  takes   the   form
\begin{equation}\label{lindepcos}
\sum\limits_{p=1}^m\xi_p\exp(ir_p\cos(\varphi-\varphi_p))=0,\quad
\varphi\in \mathbb{R}.
\end{equation}
%
%

  It  is well   known
that generating  function    for the Bessel  functions  admits the
following  representation \cite[chapter 19,\S 3]{Fikh}
\begin{equation}\label{genf}
e^{\tfrac a{2}(z-z^{-1})}=\sum\limits_{k=-\infty}^\infty
J_k(a)z^k,\qquad z\neq 0, \quad a\in\mathbb{C}.
\end{equation}
Putting in \eqref{genf} $z=ie^{i(\varphi-\varphi_p)}$ and $a=r_p,
\quad p\in\{1,..,m\}$, we arrive at the  following   expansion
into Fourier series for the functions
$f_p(\varphi)=\exp(ir_p\cos(\varphi-\varphi_p))$
\begin{equation}\label{genfcos}
\exp(ir_p\cos(\varphi-\varphi_p))=\sum\limits_{k=-\infty}^\infty
J(r_p)i^k e^{-ik\varphi_p}e^{ik\varphi}.
\end{equation}
It is easily seen that from  \eqref{lindepcos}  and
\eqref{genfcos} follows  the equality
\begin{equation*}
\sum\limits_{k=-\infty}^\infty\left[\sum\limits_{p=1}^m \xi_p
\exp(-ik\varphi_p)J_k(r_p)i^k\right]e^{ik\varphi}=0.
\end{equation*}
Therefore
\begin{equation}\label{bessel}
\sum\limits_{p=1}^m \xi_p \exp(-ik\varphi_p)J_k(r_p)=0,\quad k\in
\mathbb{N}.
\end{equation}
Using the  following   expansion into  series for $J_k(x)$
(\cite[Section 2]{Olv78})
\begin{equation*}
J_k(x)={\left(\tfrac
x{2}\right)}^k\sum\limits_{p=0}^\infty{\left(-\tfrac
{x^2}{4}\right)}^p\tfrac 1{p!\Gamma(k+p+1)},
\end{equation*}
we get $k!2^k J_k(x)=x^k[1+a_k(x)],$ where
\begin{equation*}
|a_k(x)|\leq (k+1)^{-1}\left[\exp\left(\tfrac
{x^2}{4}\right)-1\right],\quad x\in\mathbb{R},\,\, k\in\mathbb{N}.
\end{equation*}
Multiplying   the equality \eqref{bessel}  by $k!2^kR_0^{-k},$ we
obtain
\begin{equation*}
\sum\limits_{p=1}^{m'}\xi_p
\exp(-ik\varphi_p)[1+a_k(R_0)]=-\sum\limits_{p=m'+1}^m \xi_p
\exp(-ik\varphi_p){(r_pR_0^{-1})}^k[1+a_k(r_p)],
\end{equation*}
where right-hand side  equals  0 if $m'=m$. If $m'<m$, then
$r_pR_0^{-1}<1$ as $p>m'$. Thereby,
\begin{equation}\label{seq}
\lim\limits_{k\rightarrow \infty}\sum\limits_{p=1}^{m'} \xi_p
\exp(-ik\varphi_p)=0.
\end{equation}
Since arithmetic means of   the  sequence in \eqref{seq} converges
to $\xi_1(\varphi_1=0)$, then $\xi_1=0$. Thus, the  theorem is
proved.

\textbf{The second  proof.}
As  in the first  proof,  we  reduce considerations  to   investigation  of equality   \eqref{lindepcos}.
 By  uniqueness  theorem for analytic  functions,  equality  \eqref{lindepcos}
  remains  valid  for  any $z\in\C$
 \begin{equation}\label{lindepcos'}
\sum\limits_{p=1}^m\xi_p\exp(ir_p\cos(z-\varphi_p))=0,\qquad
z=x+iy\in \mathbb{C}.
\end{equation}
 Since,  by Euler formula, $\cos(z-\varphi_p)=
(e^{i(z-\varphi_p)}+e^{-i(z-\varphi_p)})/2$,  we  have
\begin{equation*}\label{mod}
|\exp(ir_p\cos(z-\varphi_p))|= \exp(-\frac{r_p}{2}\imm
(e^{i(z-\varphi_p)}+e^{-i(z-\varphi_p)}))=\exp(r_p\sh
y\sin(x-\varphi_p)).
\end{equation*}


   Denote $\psi_p(z)=\arg(\exp(ir_p\cos(z-\varphi_p)))\in [0,2\pi).$  Thus,  by
\eqref{lindepcos'},
\begin{equation}\label{assymp}
\sum\limits_{p=1}^m \xi_p\left[\exp(r_p\sh
{y}\sin(x-\varphi_p))\right]e^{i\psi_p(z)}=0.
\end{equation}
Multiplying   \eqref{assymp} by $\exp(-R_0\sh y)$, we arrive at
\begin{equation}\label{assymp1}
\sum\limits_{p=1}^{m'}\xi_p[\exp(R_0\sh
y(\sin(x-\varphi_p)-1))]e^{i\psi_p(z)}+\sum\limits_{p=m'+1}^{m}\xi_p[\exp(\sh
y(r_p\sin(x-\varphi_p)-R_0))]e^{i\psi_p(z)}=0.
\end{equation}
Setting  $x=\varphi_1+\frac{\pi}{2}$ in \eqref{assymp1}   and
passing  to the  limit as $y\rightarrow +\infty$, we obtain
$\xi_1=0$.

\end{proof}
 \begin{remark}
 The  first   proof of     Lemma \ref{lezast} belongs to  Viktor Zastavnyi. Chronologically it was obtained earlier then the second proof proposed by the other two  authors.
 \end{remark}
\begin{remark}
It might be  easily  shown   that  Lemma  \ref{lezast} is not  valid for  any  manifold  in $\R^n$.
Below  we  give   the explanatory example.
It is  obvious  that  any  hyperplane  $\pi_a$ in $\R^n$, $n\ge 2$, which  does not  contain the origin, is  given  by
\begin{equation*}
\pi_a=\{u\in \R^n:\,(u,a)=1,\quad\text{where}\quad a\in\R^n,\,\,a\ne 0\}.
\end{equation*}
  Then on such  hyperplane  the  expression   $1+\exp(i(u,\pi a))$ is identically zero.  Thus, for  any  finite  set of hyperplanes  with  the  above   property there   exists a  set  of  points
 $y_k\in\R^n$, $y_k\ne 0$, $k\in\{1,..,q\}$   such   that
 \[
  0\equiv\prod_{k=1}^q\left(1+e^{i(u,\pi y_k)}\right)=
   \sum_{p=1}^{m}\xi_p e^{i(u,x_p)},\quad  u\in \bigcup\limits_{k=1}^q\pi_{y_k}.
 \]
  Here   $m\ge 2$, $\xi_p>0,\,\, p\in\{1,..,m\}$
 and  $X=\{x_p\}_{p=1}^{m}$  is  a  subset of $\R^n$ such that $x_p\ne x_j$ as $p\ne j$.
\end{remark}

Now   we are  ready to prove the  complement of  Theorem \ref{sch}. Below we  present  two slightly  different proofs.

\begin{proof}[\textbf{The first proof of Theorem \ref{thzast}}]

  Let $\mu$ be  non-negative  finite Borel measure on  $[0,+\infty)$
from  the representation ~\eqref{schonberg} for the function  $f$.
It is obvious that  $\mu((0,+\infty))>0$ (otherwise, $f(t)\equiv
f(0)$).

Let  $X=\{x_k\}_{k=1}^{m}\subset \R^n$ and
$\xi=\{\xi_k\}_{k=1}^{m}\subset\C$ be  subsets   such that $x_p\ne
x_j$ as  $p\ne j$  and $\sum_{k=1}^m|\xi_k|>0$. From   Lemma
~\ref{lezast} with   $y=0$, $r=1$ it   follows that

 \begin{equation*}
 \int_{S_n}\left|\sum_{k=1}^{m}\xi_k
 e^{i(u,sx_k)}\right|^2\,d\nu_n(u)>0\quad \text{for  any}\quad s>0,
 \end{equation*}
 where $\nu_n$ is   the  measure from   representation ~\eqref{Omaga_n}.
 Since  $\mu((0,+\infty))>0$, we  get
\begin{equation*}
\sum_{k,j=1}^{m}\xi_k\overline{\xi}_jf(|x_k-x_j|)=
 \int_{0}^{+\infty}\left(\;
 \int_{S_n}\left|\sum_{k=1}^{m}\xi_k e^{i(u,sx_k)}\right|^2\,d\nu_n(u)
 \right)\,d\mu(s)>0\,.
\end{equation*}
 Thus, the theorem is proved.

 \textbf{The  second  proof.}  By  Theorem \ref{boch}, we have
\begin{equation*}
f(|x|)=\int_{\R^n}e^{i(u,x)}d\mu(u)\,,\quad x\in\R^n\,,
\end{equation*}
where  $\mu$  is  non-negative   finite   Borel  measure on $\R^n$.
  It is  easily  seen  that
$\supp\mu$    is a  radial  set, i.e., if $x_0\in\supp\mu$, then
the support  contains  sphere $S^r_n(0)$  of  radius  $r=|x_0|\ge 0$ centered   at the
origin. If  $f(t)\not\equiv
f(0)$, then   $\supp\mu$  contains  a sphere $S^r_n(0)$.

 Let  $f(t)\not\equiv  const$, $m\in\N$, and  the  set
$X=\{x_k\}_{k=1}^{m}\subset \R^n$  is  such   that   $x_k\ne x_j$
as $k\ne j$. Consider   the  following  quadratic  form  in $\C^m$
\begin{equation*}
Q(\xi):=\sum_{k,j=1}^{m}\xi_k\overline{\xi}_jf(|x_k-x_j|)=
 \int_{\R^n}\left|\sum_{k=1}^{m}\xi_k e^{i(u,x_k)}\right|^2\,d\mu(u)\ge
 0\,,\quad\xi=\{\xi_k\}_{k=1}^m\subset\C.
\end{equation*}
 If   $Q(\xi)=0$, then  the  function   $g(u):=\sum_{k=1}^{m}\xi_k
 e^{i(u,x_k)}$ equals  0 on $\supp\mu$   and   therefore  equals  0
 on  $S^r_n(0)$.   Finally, Lemma \ref{lezast}
yields that  $\xi_k=0,\,\,k\in\{1,..,m\}$.
 \end{proof}
  \begin{definition}
A  function
  $f(\cdot)\in C[0,\infty)\cap C^{\infty}{(0,+\infty)}$ is called   completely  monotonic  function on   $[0,\infty)$ of  the  class $M{[0,\infty)}$ if the inequality  $(-1)^kf^{(k)}(t)\ge 0$  holds for all  $k\in\mathbb{Z}_+$ and  $t>0$.
  \end{definition}
  Schoenberg noted in  \cite{Sch38_1,Sch38} that
   {\it the function  $f(\cdot)\in\bigcap\limits_{n\in\mathbb{N}}\Phi_n$  if and only if
  $f(\sqrt{\cdot})\in M[0,\infty)$.}
 Thus, it is easily implied by  Schoenberg  theorem  that $f(\cdot)\in M [0,\infty)$ yields  $f(\cdot)\in\bigcap\limits_{n\in\mathbb{N}}\Phi_n$.
  \begin{corollary}\cite[Theorem~7.14]{Wend_2005}\label{shoencor}
  If $f(\cdot)\in M[0,\infty)$ and  $f(\cdot)\not\equiv const$ on  $[0,+\infty)$,  then
 the  function  $f(|\cdot|)$  is  strictly  positive  definite on   $\mathbb{R}^n$  for any  $n\in\mathbb{N}$.
  \end{corollary}
  \begin{remark}
$(i)$ In ~\cite[Lemma~6.7]{Wend_2005}  assertion  of Lemma ~\ref{lezast}   was proven  provided that  the equality ~\eqref{lindep}  holds on a certain open subset of $\mathbb{R}$.

$(ii)$  Theorem \ref{thzast}  was  formulated in ~\cite[Theorem~6.18]{Wend_2005}  and
~\cite[Theorem~3.7]{Fasshauer_2007}  under the   additional  condition  {\it  $t^{n-1}f(t)\in L^1[0,\infty)$.}

\end{remark}

\begin{example}\label{rempoz}
Let us present  some examples  of  strictly   positive functions.
\\

$(1)$ 
 Using  the equality  $\Gamma(2p)=\tfrac{2^{2p-1}}{\sqrt{\pi}}\Gamma(p)\Gamma(p+1/2)$, we obtain
 \begin{equation*}
   \Omega_1(t)=\cos t,\quad  \Omega_2(t)=J_0(t),\quad  \Omega_3(t)={\sin t}/{t}.
 \end{equation*}
 By Theorem \ref{thzast}, the functions   $\Omega_n(s|x|)$   are  strictly  positive definite on  $\mathbb{R}^n$  for any $s>0$ and  $n\ge 2$.

 $(2)$  It  is easily  seen that the  functions  $e^{-t}$ and  $(1-e^{-t})/t=\int^1_0 e^{-ts}ds$ are  completely monotonic on  $[0,+\infty)$. Thus, by  Corollary \ref{shoencor}, the  functions     $e^{-|x|}$ and  $(1-e^{-|x|})/|x|$  are strictly positive  definite on
    $\mathbb{R}^n$ for all $n\in\mathbb{N}$.
\end{example}

\section{Three-dimensional Schr\"{o}dinger  operator with  point  interactions}


 \subsection{Boundary  triplet  and Weyl  function}
   First   we  define a boundary  triplet  for the operator $H^*$. Denote  $r_j:=\vert
x-x_j\vert,\,x\in\bR^3$, let  also
$\sqrt{\cdot}$ stands for  \, the branch of the corresponding
multifunction defined on $\C\setminus \R_+$
by the condition  $\sqrt{1}=1.$
\begin{proposition}\label{pr1}
Let  $H$ be the minimal  Schr\"{o}dinger  operator defined by   \eqref{min}
and let
$\xi_0:=\{\xi_{0j}\}_{j=1}^m,\,\,\xi_1:=\{\xi_{1j}\}_{j=1}^m\,\,\in\mathbb{C}^m.$
Then the  following  assertions hold

$(i)$    The  operator $H$  is  closed  and symmetric. The deficiency  indices  of    $H$ are $n_\pm(H)=m$.  The  defect subspace
$\mathfrak{N}_z := \mathfrak{N}_z(H)$  is
    \begin{equation}\label{4.18}
\mathfrak{N}_z =\{
\sum\limits_{j=1}^mc_j\frac{e^{i\sqrt{z}r_j}}{4\pi r_j}:\ c_j\in
\C,\ j\in \{1,\ldots,m\} \}, \quad z\in \C\setminus [0,\infty).
         \end{equation}
 $(ii)$  The adjoint  operator   $H^*$  is  given  by
\begin{gather}\label{eq3}
\dom(H^*) =\left\{ f=
\sum\limits^m_{j=1}\bigl(\xi_{0j}\frac{e^{-r_j}}{r_j}+\xi_{1j}e^{-r_j}\bigr) +
f_H:\ \xi_{0},\xi_{1}\in\bC^m,\, f_H\in\dom(H)\right\},\\
H^*f=-\sum\limits^m_{j=1}\bigl(\xi_{0j}\frac{e^{-r_j}}{r_j}+\xi_{1j}(e^{-r_j}-\frac{2e^{-r_j}}{r_j})\bigr)-\Delta f_H.\label{eq3'}
\end{gather}
$(iii)$ A totality    $\Pi =\{\kH,\Gamma_0,\Gamma_1\}$,  where
\begin{gather}
\kH=\bC^m,\quad \Gamma_0 f:=\{\Gamma_{0j}f\}^m_{j=1}=
4\pi\,\{\lim_{x\rightarrow
x_j}f(x)|x-x_j|\}^m_{j=1}=4\pi\{\xi_{0j}\}^m_{j=1},\label{g0}\\
\Gamma_1 f:=\{\Gamma_{1j}f\}^m_{j=1}=\{\lim_{x\rightarrow
x_j}\bigl(f(x)-\tfrac{\xi_{0j}}{|x-x_j|}\bigr)\}^m_{j=1},\label{g1}
\end{gather}
forms a  boundary  triplet for $H^*.$

$(iv)$ The  operator $H_0=H^*\upharpoonright \ker(\Gamma_0) (=H_0^*)$ coincides with the free Hamiltonian,
\begin{equation*}\label{freeh}
H_0 = -\Delta, \qquad \dom(H_0) = \dom(-\Delta) = W^{2,2}(\mathbb{R}^3).
\end{equation*}
\end{proposition}
          \begin{proof}
$(i)$ The  statement $(i)$ of   Proposition \ref{pr1}  is
well-known.  It  was   obtained
 in the classical  book \cite{AGHH88}(see Theorem 1.1.2).

 $(ii)$   Clearly,  $e^{-r_j}\in W^{2,2}(\mathbb{R}^3)\subset\dom(H^*)$   for  $j\in\{1,..,m\}$.

  Since  the    function  $e^{-|x|}$ is   strictly   positive  definite  on  $\mathbb{R}^3$ (Example \ref{rempoz}),
  the matrix  $(e^{-|x_k-x_j|})_{k,j=1}^m$  is positive    definite.
  Therefore    the  functions  $e^{-r_j},\,\,\, j\in\{1,..,m\}$  are  linearly independent.
Consider the  operator  $\widetilde{H}$ defined   by
     \begin{equation}\label{4.11}
\widetilde{H}:=H^*\upharpoonright\dom(\widetilde{H}),\quad
\dom(\widetilde{H})=\dom(H)\dot{+}\Span\{e^{-r_j}\}_{j=1}^m =
W^{2,2}(\mathbb{R}^3).
     \end{equation}
Since $\dom(\widetilde{H}) = \dom(-\Delta) = W^{2,2}(\mathbb{R}^3)$
and both operators $\widetilde{H}$ and $-\Delta$ are proper
extensions of $H$, we have $\widetilde{H} = -\Delta$ and the
operator $\widetilde{H}$ is self-adjoint.

Further, since the functions $\frac{e^{-r_j}}{4\pi r_j}\in
\mathfrak{N}_{-1},\,\,j\in\{1,..,m\}$ are linearly independent,
and, by the second J. von Neumann formula,
\[
\dim(\dom (H^*)/\dom(\widetilde{H}))=n_\pm(H)=m,
\]
representation  \eqref{eq3} is proved.

$(iii)$  Let  $f,g\in\dom(H^*)$. By  \eqref{eq3}, we  have
 \begin{equation*}\label{fg}
  f= \sum\limits^m_{k=1}f_k + f_H, \,\,\, f_k =\xi_{0k}\,\frac{e^{-r_k}}{
r_k}+\xi_{1k}e^{-r_k}, \,\,\,\,
  g = \sum\limits^m_{k=1}g_k + g_H, \,\,\, g_k =
\eta_{0k}\,\frac{e^{-r_k}}{r_k}+\eta_{1k}\,e^{-r_k},
 \end{equation*}
 where $f_H,
g_H\in\dom(H)$, and  $\xi_{0k},\ \xi_{1k}, \eta_{0k},\
\eta_{1k}\in\bC,\ k\in\{1,..,m\}$.

 Applying  \eqref{g0}-\eqref{g1}  to $f,g\in\dom(H^*)$, we obtain
\begin{multline}\label{tripl}
\Gamma_0f=4\pi\{\xi_{0j}\}_{j=1}^m,\quad\Gamma_1f=\left\{-\xi_{0j}+\sum\limits_{k\neq
j}\xi_{0k}\frac{e^{-|x_j-x_k|}}{|x_j-x_k|}+\sum\limits_{k=1}^m\xi_{1k}e^{-|x_j-x_k|}\right\}_{j=1}^m,\\
\Gamma_0g=4\pi\{\eta_{0j}\}_{j=1}^m,\quad\Gamma_1g=\left\{-\eta_{0j}+\sum\limits_{k\neq
j}\eta_{0k}\frac{e^{-|x_j-x_k|}}{|x_j-x_k|}+\sum\limits_{k=1}^m\eta_{1k}e^{-|x_j-x_k|}\right\}_{j=1}^m.\quad\quad\qquad\quad
\end{multline}
It is easily seen that
\begin{multline*}
(H^*f,g) -(f,H^*g) =
\sum\limits^m_{k,j=1}\biggl(\left(\xi_{0j}H^*\left(\frac{e^{-r_j}}{
r_j}\right),\eta_{1k}\,e^{-r_k}\right) - \left(\xi_{0j}\,\frac{e^{-r_j}}{
r_j},\eta_{1k}H^*(e^{-r_k})\right)\\+\left(\xi_{1j}H^*(e^{-r_j}),\eta_{0k}\,\frac{e^{-r_k}}{r_k}\right)
-\left(\xi_{1j}\,e^{-r_j},\eta_{0k}H^*\left(\frac{e^{-r_k}}{r_k}\right)\right)\biggr).
\end{multline*}
 Using the  second Green  formula,  we get
\begin{multline}\label{green'}
\left(H^*\left(\frac{e^{-r_j}}{r_j}\right),e^{-r_k}\right)-
\left(\frac{e^{-r_j}}{r_j},H^*(e^{-r_k})\right)\\=\lim\limits_{r\rightarrow\infty}\biggl(\int\limits_{B_r(x_j)\backslash
B_{\tfrac1{r}}(x_j)}-\Delta\left(\frac{e^{-r_j}}{r_j}\right)e^{-r_k}\mathrm{dx}+\int\limits_{B_r(x_j)\backslash
B_{\tfrac1{r}}(x_j)}\frac{e^{-r_j}}{r_j}\Delta(e^{-r_k})\mathrm{dx}\biggr)\\
=\lim\limits_{r\rightarrow\infty}\int\limits_{S_r(x_j)}\biggl(-\frac{\partial}{\partial n}\left(\frac{e^{-r_j}}{r_j}\right)e^{-r_k}+\frac{e^{-r_j}}{r_j}\frac{\partial {(e^{-r_k})}}{\partial
n}\biggr)\mathrm{ds}\\+\lim\limits_{r\rightarrow\infty}
\int\limits_{S_{\tfrac1{r}}(x_j)}\biggl(\frac{\partial}{\partial n}\left(\frac{e^{-r_j}}{r_j}\right)e^{-r_k}-\frac{e^{-r_j}}{r_j}\frac{\partial  {(e^{-r_k})}}{\partial
n}\biggr)\mathrm{ds}= -4\pi e^{-|x_j-x_k|},
\end{multline}
where ${n}$   stands   for the exterior  normal  vector to
$S_r(x_j)$  and $S_{\tfrac1{r}}(x_j)$,  respectively.

 Indeed, noticing that  $\frac{\partial}{\partial n}\left(\frac{e^{-r_j}}{r_j}\right)=-\frac{e^{-r_j}}{r_j}(1+\tfrac1{r_j})$, we can   easily  show  that the first  integral in  the formula \eqref{green'} tends   to  zero  as $r$  tends to  infinity.
Further,

\begin{multline*}
\lim\limits_{r\rightarrow\infty}
\int\limits_{S_{\tfrac1{r}}(x_j)}\frac{\partial}{\partial n}\left(\frac{e^{-r_j}}{r_j}\right)e^{-r_k}\mathrm{ds}
=\lim\limits_{r\rightarrow\infty}
\int\limits_{S_{\tfrac1{r}}(0)}\frac{\partial}{\partial n}\left(\frac{e^{-|y|}}{|y|}\right)e^{-|y+x_j-x_k|}\mathrm{ds}\\=-\lim\limits_{r\rightarrow\infty}\left[\frac{4 \pi}{r^2}e^{-1/r}r(1+r)e^{-|y^*+x_j-x_k|}\right]=-4\pi\lim\limits_{y^{*}\rightarrow
0}e^{-|y^*+x_j-x_k|}=-4\pi e^{-|x_j-x_k|},
\end{multline*}
and
\begin{multline*}
\lim\limits_{r\rightarrow\infty}\int\limits_{S_{\tfrac1{r}}(x_j)}\frac{e^{-r_j}}{r_j}\frac{\partial  {(e^{-r_k})}}{\partial
n}\mathrm{ds}=\lim\limits_{r\rightarrow\infty}\int\limits_{S_{\tfrac1{r}}(0)}\frac{e^{-|y|}}{|y|}\frac{\partial  {(e^{-|y+x_j-x_k|})}}{\partial
n}\mathrm{ds}\\=\lim\limits_{r\rightarrow\infty}\left[\frac{4\pi}{r^2} e^{-1/r}r \frac{\partial  {(e^{-|y+x_j-x_k|})}}{\partial
n}|_{y=y^{**}}\right]=0, \quad  y^*\in S_{\tfrac1{r}}(0),\,\, y^{**}\in S_{\tfrac1{r}}(0).
\end{multline*}


 Finally, by \eqref{tripl},
\begin{multline*}
(H^*f,g)
-(f,H^*g)=4\pi\sum\limits^m_{k,j=1}\biggl(-\xi_{0j}
\overline{\eta}_{1k}e^{-|x_j-x_k|}+
\xi_{1j}\overline{\eta}_{0k}e^{-|x_j-x_k|}\biggr)\\=
\sum\limits_{j=1}^m(\Gamma_{1j}f,\Gamma_{0j}g)-(\Gamma_{0j}f,\Gamma_{1j}g)=(\Gamma_1f,\Gamma_0g)-(\Gamma_0f,\Gamma_1g).
\end{multline*}
 Thus, the Green identity  is satisfied.
 Let us  show  that    that  the mapping
 $\Gamma=(\Gamma_0,\Gamma_1)^\top$ is surjective. Let $(h_0,h_1)^\top\in\cH\oplus\cH$,
 where $h_0=\{h_{0j}\}_{j=1}^m,h_1=\{h_{1j}\}_{j=1}^m$
 are  vectors from $\mathbb{C}^m$.  According  to   \eqref{eq3},   any  $f\in\dom(H^*)$   admits   the   representation   $f=f_H+\sum\limits^m_{j=1}\bigl(\xi_{0j}\frac{e^{-r_j}}{r_j}+\xi_{1j}e^{-r_j}\bigr)$.
  Let us  put
\begin{equation}\label{hat}
E_0:=\left(\frac{e^{-|x_k-x_j|}}{|x_k-x_j|-\delta_{kj}}\right)_{j,k=1}^m,\quad
E_1:=\left(e^{-|x_k-x_j|}\right)_{k,j=1}^m,
\end{equation}
where  $\delta_{kj}$  denotes\,the Kronecker  symbol.
  Invertibility of the matrix $E_1$ has  been already   established  above. Therefore setting  $\xi_0=\tfrac1{4\pi}h_0$ and
$\xi_1=E_1^{-1}(h_1-\tfrac1{4\pi}E_0h_0)$, we get  $\Gamma_0f=h_0$ and
$\Gamma_1f=h_1$. Thus, $\Pi$  is  a boundary  triplet for $H^*$.

$(iv)$  This statement  easily   follows  from \eqref{eq3},
\eqref{g0} and \eqref{4.11}.
\end{proof}

Combining   Proposition  \ref{s-aext} 
with  formulas  \eqref{eq3}, \eqref{tripl}, we  arrive at  the
following    proposition.
   \begin{proposition}
Let   $H$  be the  minimal Schr\"odinger  operator  defined by  \eqref{min},
$\Pi =\{\kH,\Gamma_0,\Gamma_1\}$ a boundary triplet for $H^*$
defined by  \eqref{g1} and  let the matrices $E_0, E_1$ be defined
by \eqref{hat}.
Then the   set of  self-adjoint  extensions $\widetilde H\in \Ext_H$  is  
parameterized   as
     \begin{equation}\label{s-adjconcr}
\widetilde H = H_\Theta=H^*\upharpoonright\{
f=\sum\limits^m_{j=1}\bigl(\xi_{0j}\frac{e^{-r_j}}{r_j} +
\xi_{1j}e^{-r_j} \bigr) +
f_H\in\dom(H^*):\,\{\Gamma_0f,\Gamma_1f\}\in\Theta\},
  \end{equation}
 where $\Theta$  runs  through the  set of all  self-adjoint  linear relations in $\kH$.  Moreover,
  $\Theta_{\infty}$ and
 $\Theta_{\op}$ are defined by
%
%
    \begin{gather}\label{op3}
\Theta_{\infty} = \{0, \kH_\infty\} = \{\{0, E_1\xi''_1\}:\,\xi''_1\perp E_1 \xi_0,\  \xi_0\in \kH_{\op} \}, \\  
\Theta_{\op} = \{\{4\pi \xi_0, E_0\xi_0 + E_1\xi'_1\}:\, \xi_0\in
\kH_{\op},\,\xi'_1 = E_1^{-1}(4\pi B\xi_0 - E_0\xi_0)\}, \label{mul3}
  \end{gather}
where $B=B^*\in [\kH_{\op}]$. In particular, $\widetilde H = H_\Theta$
is disjoint with $H_0$ if and only if  $\dom(\Theta)=\mathbb{C}^m$.
In this case $\Theta=\Theta_{\op}= grB.$
\end{proposition}
        \begin{proof}
By \eqref{s-aext}, any    self-adjoint extension  $H_\Theta$  is
parameterized  as follows
  \begin{multline}\label{rel}
H_\Theta=H^*\upharpoonright \dom(H_\Theta),\quad  \\   
 \dom(H_\Theta)=\{f=\sum\limits^m_{j=1}\bigl(\xi_{0j}\frac{e^{-r_j}}{r_j}+\xi_{1j}e^{-r_j}
\bigr)+f_H:\,\{4\pi\xi_0,E_0\xi_0+E_1\xi_1\}\in\Theta=\Theta^*\}.
   \end{multline}
It is easily seen   that   self-adjointness    of the  relation
$\Theta$ is  equivalent to  the   condition
$(E_1\xi_1,\xi_0)={(\xi_0, E_1\xi_1)}$. Since $\kH_\infty =
\mul(\Theta)\perp\dom(\Theta)= \kH_{\op}$, this  condition is
equivalent to \eqref{op3}-\eqref{mul3}. For instance, it follows
from \eqref{op3} and \eqref{mul3} that $(E_1\xi_1'',\xi_0) =0
=(\xi_0, E_1\xi_1'')$ and
   \begin{equation*}
(E_1\xi'_1,\xi_0)=(4\pi B\xi_0 - E_0\xi_0,\xi_0)=(\xi_0,4\pi
B\xi_0 - E_0\xi_0)=(\xi_0,E_1\xi'_1).
   \end{equation*}
Hence $(E_1\xi_1,\xi_0)=(\xi_0,E_1\xi_1)$. The arguments can be
reversed.
      \end{proof}
      \begin{remark}
Note that    the $m$-parametric family $H^{(3)}_{X,\alpha}$
investigated  in \cite[Theorem 1.1.1]{AGHH88} is parameterized by
the set of  diagonal matrices $B_\alpha :=
\diag(\alpha_1,..,\alpha_m) = B_\alpha^*$.   
Namely, in   this  case, we  put in  \eqref{rel}
$\Theta=\Theta_{\op}=B_\alpha$ and
%
%
\begin{equation}\label{fam}
H^{(3)}_{X,\alpha}=H^*\upharpoonright\{
f=\sum\limits^m_{j=1}\xi_{0j}\frac{e^{-r_j}}{r_j}+\sum\limits^m_{k,j=1}
b_{jk}(\alpha)\xi_{0k}e^{-r_j}+f_H\},\quad
\end{equation}
where $\widetilde{B}_\alpha
=(b_{jk}(\alpha))_{j,k=1}^m=E_1^{-1}(4\pi B_\alpha-E_0)$.
  \end{remark}
     \begin{proposition}\label{pr2}
Let   $H$  be the minimal  Schr\"{o}dinger  operator defined by  \eqref{min}
and let  $\Pi=\{\kH,\Gamma_0,\Gamma_1\}$ be  the  boundary triplet
for $H^*$ defined by \eqref{g0}-\eqref{g1}.  Then

$(i)$  the corresponding Weyl
function $M(\cdot)$  is
    \begin{equation}\label{W3}
 M(z)=\left(\tfrac{i\sqrt{z}}{4\pi}\delta_{jk}+\widetilde{G}_{\sqrt{z}}(x_j-x_{k})\right)_{j,k=1}^m, \qquad
 z\in\mathbb{C}_+,
   \end{equation}
    where
$\widetilde{G}_{\sqrt{z}}(x)=\left\{%
  \begin{array}{ll}
    \frac{e^{i\sqrt{z}|x|}}{4\pi|x|}, & \hbox{x$\neq$0;} \\
    0, & \hbox{x=0.} \\
\end{array}%
\right.$ and $\delta_{kj}$ denotes the  Kronecker symbol;

$(ii)$  the corresponding $\gamma(\cdot)$-field is given by
\[ \gamma(z)\{a_j\}^m_{j=1} =
\sum\limits^m_{j=1}a_j\,\frac{e^{i\sqrt{z}r_j}}{4\pi r_j}.\]
 \end{proposition}
    \begin{proof}
 Let  $f_z\in \mathfrak{N}_z.$
 By  \eqref{4.18},  $f_z = \sum\limits^m_{j=1}a_j\,\frac{e^{i\sqrt{z}r_j}}{4\pi r_j},\
a_j\in\bC$.
 Applying   $\Gamma_0$  and  $\Gamma_1$ to  $f_z$, we  get
 \begin{equation}\label{eq*}
\Gamma_0f_z = \{a_j\}^m_{j=1},\qquad
 \Gamma_1f_z=\left\{a_j\frac{i\sqrt{z}}{4\pi}+\sum\limits_{k\neq
j}a_k\frac{e^{i\sqrt{z}|x_j-x_k|}}{4\pi|x_j-x_k|}\right\}^m_{j=1},
\quad z\in \C\setminus [0,\infty).
    \end{equation}
Substituting these formulas in \eqref{2.3A} (cf. Definition
\ref{Weylfunc}), we arrive at  \eqref{W3}.

The second statement  follows  immediately   from   \eqref{eq*} and
\eqref{2.3A}.
   \end{proof}
\begin{remark}
$(i)$    In  the  case  $m=1$, description    \eqref{4.18}  of $\mathfrak{N}_z$ was   obtained by
Lyantse and Majorga   \cite [Lemma 4.1]{LyaMaj}.

 $(ii)$  The  first   construction  of the   boundary    triplet, in the  case  $m=1$,
goes  apparently back  to  the  paper   by  Lyantse  and Majorga
   \cite[Theorem 2.1]{LyaMaj}.  Slightly different  boundary triplet was obtained in  \cite[Section 5.4]{bmn}.
    Another construction  of the    boundary   triplet  in  the  situation   of general  elliptic  operator  with  boundary   conditions  on   a set  of zero  Lebesgue measure   was   obtained     by  A. Kochubei  \cite{Koc82}.  However   this  construction is  not  suitable  for our purpose.

 $(iii)$  The  Weyl  function in the form \eqref{W3}  has  appeared
 in Krein's formula  for   resolvent of  $H_{X,\alpha}^{(3)}$ (see  \cite[chapter II.1]{AGHH88}). 
     In this connection  we also mention   the  paper  by  Posilicano \cite[Example  5.3]{Pos08}.
Note also that,  in the  case  $m=1$, the  Weyl  function   was
computed   in \cite[Section 5.4]{bmn}  and   without boundary  triplets  in
\cite[Section 10.3]{Ash10}.
    \end{remark}


\subsection{ Spectrum of the  self-adjoint  extensions of the  minimal Schr\"{o}dinger  operator $H$}
In what follows  we need the following lemma.

\begin{lemma}\label{zero}
Let  $H^*$  be  defined  by \eqref{eq3}-\eqref{eq3'}. Then
$0\notin\sigma_p(H^*).$
\end{lemma}
     \begin{proof}
 Let $f_0\in\dom(H^*)$  and $H^*f_0 =0$. Observing   that $\frac{e^{-r_j}-1}{r_j}\in W^{2,2}_{\loc}(\mathbb{R}^3)$
 for any $j\in\{1,..,m\}$, we obtain from   \eqref{eq3} that  $f_0$ admits
a representation
\begin{equation}\label{harm}
f_0 =\sum\limits_{j=1}^m \xi_{0j}\frac{1}{r_j} + g, \quad\text{with }\quad g:=\sum\limits_{j=1}^m\bigl(\xi_{0j}\frac{e^{-r_j}-1}{r_j}+\xi_{1j}e^{-r_j}\bigr)+f_H\in  W^{2,2}_{\loc}(\mathbb{R}^3).
\end{equation}
  By  definition of  $H^*$,  we  get $(f_0,\Delta\varphi)=0,\,\,\varphi\in C_0^\infty(\mathbb{R}^3\setminus  X)$,   i.e.,  $f_0$   is a week  solution of the   equation $\Delta f_0=0.$ By  regularity   theorem (Weyl's lemma \cite[chapter 8]{Gil88}), $f_0(\cdot)$  is  harmonic  function in $\mathbb{R}^3\setminus X$. Since   $(\Delta\frac 1{r_j})(x)=0$    and  $(H^*f_0)(x)=-(\Delta f_0)(x)=0$     for  $x\notin X$,  we  get that $g(x)$
is  harmonic function for    $x\notin X$  and continuous on
$\mathbb{R}^3$  by Sobolev embedding  theorem.

It  follows from  definition \eqref{harm} that  $g(\cdot)$  is  bounded.
Therefore, by     desingularization theorem (see \cite[chapter
IV,\S 3]{mikh76}),  it  can be extended by  continuity to  $X$,
and the extended  function is  harmonic on $\mathbb{R}^3$.
Therefore, by the  Liouville theorem for  harmonic  functions,
$g(x)=const,\,\, x\in \mathbb{R}^3.$ Since   $g(\cdot)\in
L^2(\mathbb{R}^3)$,  we  have  $g(\cdot)\equiv 0$  in
$\mathbb{R}^3$. Observing that    $1/ r_j\notin
L^2(\mathbb{R}^3),\,\, j\in\{1,..,m\}$  and that the functions
$1/r_j$  are linearly independent,   we have   $f_0(\cdot)\equiv
0$ in  $\mathbb{R}^3$.
\end{proof}
In  the  following  theorem  we describe   spectrum of the
self-adjoint extensions of $H$.
\begin{theorem}\label{spec3}
Let  $H$ be  the minimal  Schr\"{o}dinger  operator defined by \eqref{min} and  $\Pi$ be  the boundary  triplet for  $H^*$ defined by
\eqref{g0}-\eqref{g1}. Let also   $M(\cdot)$ be    the  corresponding
Weyl function defined by  \eqref{W3}. Assume that
$H_\Theta=H_\Theta^*\in Ext_H$  defined  by \eqref{s-adjconcr}. Then the   following   assertions
hold.

$(i)$ The extension $H_\Theta$ has  purely absolutely continuous
non-negative spectrum   of infinite  multiplicity.

 $(ii)$ Point  spectrum of the  extension  $H_\Theta$  consists  of at  most   $m$
 negative  eigenvalues (counting multiplicities), $\kappa_-(H_\Theta)\le m.$ Moreover,
$z\in\sigma_p(H_\Theta)\cap\bR_-\,$ if and only if\,
$0\in\sigma_p(\Theta-M(z))$, i.e.,
     \begin{equation}\label{4.27}
z\in\sigma_p(H_\Theta)\cap\bR_-
\Longleftrightarrow\,0\in\sigma_p\bigl(\Theta-\left(\tfrac{i\sqrt{z}}{4\pi}\delta_{jk}
+\widetilde{G}_{\sqrt{z}}(x_j-x_{k})\right)_{j,k=1}^m\bigr).
\end{equation}
The corresponding   eigenfunctions $\psi_z$ have the form
\begin{equation}\label{eigen}
\psi_z=\sum\limits_{j=1}^m a_j\frac{e^{i\sqrt{z}r_j}}{4\pi
r_j},\quad \text{where} \quad  (a_1,.., a_m)^\top\in \ker(\Theta-M(z)).  
      \end{equation}
%
%
  \newline
 $(iii)$ The number of  negative eigenvalues of
 the  self-adjoint  extension $H_\Theta$ equals  the number  of
 negative eigenvalues of  the  relation
$\Theta-M(0)$, $\kappa_{-}(H_\Theta)= \kappa_{-}(\Theta-M(0))$,
i.e.,
\begin{equation*}
\kappa_{-}(H_\Theta)=\kappa_{-}\left(\Theta-
\left(\frac{1-\delta_{jk}}{4\pi|x_k-x_j|+\delta_{jk}}\right)_{j,k=1}^m\right).
\end{equation*}
\end{theorem}
\begin{proof}
$(i)$
 Note  that symmetric operator  $H$   is not simple  since     the
multiplicity of the spectrum of its  extension  $H_0=H_0^*$  is infinite.
Therefore $H$   admits  the  representation $H =
\widehat{H}\oplus H_s$,   where $\widehat{H}$   and  $H_s$  are  the simple  and   the  self-adjoint part of  $H$, respectively,  defined  by
 \begin{gather*}
 \widehat{H}=H\upharpoonright P_{\widehat{\mathfrak{H}}}(\dom (H)),\quad  \widehat{\mathfrak{H}}=\overline{\Span}\{\mathfrak{N}_z: z\in\mathbb{C}\setminus \mathbb{R}\},\\
 H_s=H\upharpoonright P_{\mathfrak{H}_s}(\dom (H))=H_s^*,\qquad  \mathfrak{H}_s=L^2(\mathbb{R}^3)\ominus  \widehat{\mathfrak{H}},
 \end{gather*}
 where $ P_{\widehat{\mathfrak{H}}}$  and  $P_{\mathfrak{H}_s}$  are   orthogonal  projectors  onto $\widehat{\mathfrak{H}}$  and $\mathfrak{H}_s,$  respectively.
  Clearly,   a  totality $\widehat{\Pi}=\{\mathcal{H},\widehat{\Gamma}_0,\widehat{\Gamma}_1\}=:\{\mathcal{H},\Gamma_0\upharpoonright \widehat{\mathfrak{H}},\Gamma_1\upharpoonright \widehat{\mathfrak{H}}\}$   forms a  boundary  triplet   for    $\widehat{H}^*$. Then  the
operator $H_0$ takes the form $H_0=\widehat {H}_0\oplus H_s$, where
$\widehat {H}_0=\widehat H^*\upharpoonright\ker(\widehat{\Gamma}_0)=\widehat H_0^*$.

 For\,\,simplicity, we 
confine  ourselves  to the  case  of realizations   $H_\Theta$  disjoint  with  $H_0$, i.e.,  $\dom(H_\Theta)\cap
\dom(H_0)=\dom(H)$. Then
$H_\Theta=H_B=H^*\upharpoonright\ker(\Gamma_1-B\Gamma_0)$ with
$B=B^*\in[\mathcal{H}]$. Thereby  $H_B=\widehat{H}_B\oplus H_s$,\,\,  $\widehat{H}_B=\widehat H^*\upharpoonright \ker(\widehat{\Gamma}_0^B)$, where corresponding boundary triplet $\widehat{\Pi}_B=\{\mathcal{H}^B, \widehat{\Gamma}_0^B,\widehat{\Gamma}_1^B\}$  is  defined  by
\[
\mathcal{H}^B=\mathcal{H},\quad\widehat{\Gamma}_0^B=B\widehat{\Gamma}_0-\widehat{\Gamma}_1,\quad\widehat{\Gamma}_1^B=\widehat{\Gamma}_0.
\]
The  appropriate Weyl  function  is
$M_B(z)=(B-M(z))^{-1}$.
Moreover, it is  easily  seen  that
\begin{equation}\label{imcompl}
\imm (M_B(z))=(B-M(z))^{-1}\imm (M(z))(B-M^*(z))^{-1}, \quad  z\in\mathbb{C}\setminus\sigma_p(H_B).
\end{equation}
It   follows from  \eqref{W3}   that the   strong limit $M(x+i0)=s-\lim\limits_{y\downarrow
0}M(x+iy)$ exists for all $x\in\bR$  and
 \begin{equation*}
  M(x+i0)=\left(\frac{i\sqrt{x}}
{4\pi}\delta_{kj}+\frac{e^{i\sqrt{x}|x_k-x_j|}-\delta_{kj}}{4\pi|x_k-x_j|+\delta_{kj}}\right)_{j,k=1}^m,\quad x\in\mathbb{R}.
\end{equation*}
Therefore
\begin{equation}\label{im}
 \imm (M(x+i0))=\left(\frac{\sqrt{x}}
{4\pi}\delta_{kj}+\frac{\sin(\sqrt{x}|x_k-x_j|)}{4\pi|x_k-x_j|+\delta_{kj}}\right)_{j,k=1}^m,\quad x\in\mathbb{R}_+
\end{equation}
and  $\imm (M(x+i0))=0$  for  $x\leq 0$.  Combining  this   fact  with   \eqref{imcompl},  we  conclude  that
 \begin{equation}\label{imreal}
 \imm(M_B(x+i0))=(B-M(x+i0))^{-1}\imm (M(x+i0))(B-M^*(x+i0))^{-1},\quad x\in\mathbb{R}\setminus\sigma_p(H_B).
 \end{equation}
 %
 %

  Since  the  functions  $\frac{\sin s
 x}{sx}\in\Phi_3,\,\,s>0$    (see  Example  \ref{rempoz}),  we conclude  that  the   matrix $\imm
 (M(x+i0))/\sqrt{x}$  is   positive  definite  for all $x>0.$
 Hence the
 matrix   $\imm (M_B(x+i0))/ \sqrt{x}$  is also  positive definite for  every
 $x>0$.  Thereby,  the matrix  $B-M(x+i0)$  is   invertible  for
 all $x\in\mathbb{R}_{+}$  and consequently $\sigma_p(\widehat{H}_B)\cap\mathbb{R}_+=\emptyset.$  It    also   follows from \eqref{imreal} that  the   multiplicity  function $d_{M_B(x)}$  is   maximal  for all  $x>0$, i.e., $d_{M(x)}=d_{M_B(x)}=m$.

 Therefore,
  by Proposition \ref{ac}$(ii)$, $\widehat{H}_0^{ac}$ and $\widehat{H}_B^{ac}$ are
unitarily equivalent. By  Proposition  \ref{ac}$(i),(ii)$, $\sigma_{sc}(\widehat{H}_0)=\sigma_{p}(\widehat{H}_0)=\emptyset$  and
$\sigma_{ac}(\widehat{H}_0)=[0,\infty)$.
Since $H_s=H_s^{ac}$ and $\sigma_{ac}(H_s)=[0,\infty)$(see
\cite[Chapter XIII]{ReeSim78}), 
    $\sigma_{ac}(H_B)=[0,\infty)$  and  $\sigma_p(H_B)\cap\mathbb{R}_+=\emptyset.$
Further, Proposition  \ref{ac}$(i)$ and  the  equality $H_s=H_s^{ac}$
together  yield  $\sigma_{sc}(H_B)\cap \R_+=\emptyset$.  Absence  of  negative  continuous spectrum of  $H_B$ follows   from the relations  $\widehat{H}\geq 0$  and $n_\pm(\widehat{H})=m$ (see \cite[chapter 4,\S14]{Nai}). To complete  the proof it remains  to  apply  Lemma \ref{zero}.

$(ii)$ 
According to the decomposition $H = \widehat{H}\oplus H_s$, we
have $H_\Theta = \widehat{H}_\Theta\oplus H_s.$  Since $\R_-
\subset\rho(H_0)$ and $H_s\ge 0$, Proposition
\ref{prop_II.1.4_spectrum} applied to the simple part
$\widehat{H}$ of $H$  yields the equivalences
  %
  %
\[
z\in\sigma_p({H}_\Theta)\cap \bR_-\Longleftrightarrow
z\in\sigma_p(\widehat{H}_\Theta)\cap \bR_-\Longleftrightarrow
0\in\sigma_p(\Theta-M(z)).
\]
Combining this formula with formula \eqref{W3} for the Weyl
function  yields \eqref{4.27}.  Formula \eqref{eigen}  follows  from  \eqref{eq*} and  Proposition
\ref{prop_II.1.4_spectrum}(ii) applied to the simple part
$\widehat{H}$ of $H.$

It remains to note that the inequality $\kappa_-(H_\Theta) =
\kappa_-(\widehat{H}_\Theta)\le m$ is immediate from the  fact
that $H\geq 0$ and $n_{\pm}(H)= n_{\pm}(\widehat{H}) = m$ (see \cite[chapter 4,\S14]{Nai}).

$(iii)$  Combining
Proposition \ref{prkf}$(ii)$  with  \eqref{W3}, we   arrive at $(iii).$
  \end{proof}
\begin{remark}
Note that the  invertibility of the  matrix $\imm (M(x+i0))$
 \eqref{im}  for all  but  finite  number of   $x\in\mathbb{R}_+$  can directly  be  extracted
without   involving positive definite functions theory.
  Clearly, the function $V(x):=\det(\imm (M(x)))/\sqrt{x}$
admits holomorphic continuation on $\mathbb{C}$. 
 Since    $\lim\limits_{z\rightarrow \infty}V(z)=I_m$, the number
 of  zeroes of  $V(\cdot)$  on $\bR_{+}$ is finite  because  of  its   analiticity.
\end{remark}
\begin{remark}
$(i)$ The   description  of   absolutely   continuous    and   point
spectrum  in the  particular  case  of  the  family
$H_{X,\alpha}^{(3)}$ defined by \eqref{fam} was  obtained  in
\cite[Theorem 1.1.4]{AGHH88}  by  another  method.

$(ii)$ Complete   description of  the negative   spectrum of the $m$-parametric  family  $H_{X,\alpha}^{(3)}$  was recently   obtained   by  O.~Ogurisu in \cite[section 2]{ogu10} by  another method.
\end{remark}

\subsection{Non-negative self-adjoint  extensions of the  minimal Schr\"{o}dinger  operator $H$}
Here we propose slightly  different   boundary triplet for $H^*$ and compute
the corresponding Weyl function. It turns out that   the new Weyl function is more convenient for the   description  of   non-negative self-adjoint extensions  of the  minimal  operator $H$ than the one
constructed   in Proposition \ref{pr2}.
%
%

\begin{proposition}
Let  $H$ be  the minimal  Schr\"{o}dinger  operator defined by  \eqref{min},
let $\Pi$ be  the boundary  triplet for  $H^*$ defined by
\eqref{g0}-\eqref{g1}, and let $M(\cdot)$  be  the  corresponding
Weyl function defined by  \eqref{W3}. Then the  set  of all non-negative self-adjoint   extensions  $H_\Theta\in Ext_H$  is  parameterized  by
\begin{equation*}
H_\Theta=H^*\upharpoonright\left\{f=\sum\limits^m_{k,j=1}b'_{jk}\xi_{1k}\frac{e^{-r_j}}{r_j}+\sum\limits_{j=1}^m\xi_{1j}e^{-r_j}
+ f_H,\,\,f_H\in \dom(H)\right\},
\end{equation*}
where
$B'=(b'_{kj})_{k,j=1}^m=\frac1{4\pi}BE_1$,
with  $E_1=\left(e^{-|x_k-x_j|}\right)_{k,j=1}^m$  and  $B$ runs over  the  set of all  matrices  satisfying the     condition
\quad $0<B<  4\pi \left(\left(\frac{1-e^{-|x_k-x_j|}-\delta_{jk}}{|x_k-x_j|-\delta_{jk}}\right)_{k,j=1}^m\right)^{-1}$.
\end{proposition}
\begin{proof}
Alongside the triplet $\Pi=\{\cH,\Gamma_0, \Gamma_1\}$ consider
the new  boundary  triplet $\widetilde{\Pi}=\{\widetilde{\cH},\widetilde{\Gamma}_0, \widetilde{\Gamma}_1\}$ (cf. \cite{DerMal91}),
%
\begin{equation*}
\widetilde{\cH}:=\cH,\quad\widetilde{\Gamma}_0:=\Gamma_0,\quad\widetilde{\Gamma}_1:=\Gamma_1-\tfrac
1{4\pi}E_0\Gamma_0,
\end{equation*}
 where $E_0$  is defined  by \eqref{hat}.
Then $\widetilde{M}(z)=M(z)-\tfrac 1{4\pi}E_0$.  Using \eqref{W3}, we  obtain
 \begin{equation*}
\widetilde{M}(0)=\tfrac1{4\pi}\left(\frac{1-e^{-|x_k-x_j|}-\delta_{jk}}{|x_k-x_j|-\delta_{jk}}\right)_{k,j=1}^m.
 \end{equation*}
 By  Proposition \ref{prkf}, non-negative self-adjoint  extensions $H_\Theta$ are described   by
the condition $\Theta-\widetilde{M}(0)\geq 0$. Since the function
$f(t)=\frac{1-e^{-t}}{t}\in \bigcap\limits_{n\in \mathbb{N}}\Phi_n $ (see
Example \ref{rempoz}), the matrix $\widetilde{M}(0)$ is positive
definite. Therefore $\Theta$ is  also  positive   definite  and inverse matrix $\Theta^{-1}=B$  exists  and  $H_\Theta=H^*\upharpoonright\ker(B\widetilde{\Gamma}_1-\widetilde{\Gamma}_0)$.  Thus,  by \eqref{s-aext}  and \eqref{eq3}, the  desired  parametrization holds.
\end{proof}
\begin{remark}
It should  be noted  that   the above  description is  close  to  the following   obtained   by  Yu. Arlinskii  and  E. Tsekanovskii (see
\cite[Theorem 5.1]{ArlTse05})  in   the framework  of another  approach. Namely, any  non-negative  self-adjoint  extension  $\widetilde{H}$  of the minimal  operator  $H$  admits the  representation
\begin{multline*}
\dom(\widetilde{H})=\left\{f_H+\sum\limits_{j=1}^m\xi_j\sqrt{\tfrac\pi{2}}\frac{e^{-\tfrac {r_j}{\sqrt{2}}}}{r_j}\sin(\tfrac {r_j}{\sqrt{2}})+
\sum\limits_{k,j=1}^mu_{kj}\xi_k\sqrt{\tfrac\pi{2}}\frac{e^{-\tfrac {r_j}{\sqrt{2}}}}{r_j}\cos(\tfrac {r_j}{\sqrt{2}})\right\},\\
f_H\in \dom(H),\quad
(\xi_1,..,\xi_m)\in\mathbb{C}^m,\qquad\qquad\qquad\qquad\qquad\quad
\end{multline*}
where $\mathcal{U}=(u_{kj})_{k,j=1}^m$ runs over  the set of
matrices satisfying  the condition  $0\leq
\mathcal{U}\mathcal{G}\leq
\mathcal{G}\mathcal{W}_0^{-1}\mathcal{G}$ with  $\mathcal{W}_0$   and    $\mathcal{G}$  defined by,  respectively,
\begin{gather*}
2\pi^2\left(\frac{\delta_{jk}}{\sqrt{2}}+\frac{1-\exp(\frac{-r_{kj}}{\sqrt{2}})\cos\frac{r_{kj}}{\sqrt{2}}}{r_{kj}+\delta_{jk}}\right)_{k,j=1}^m,\,\,
 2\pi^2\left(\frac{\delta_{jk}}{\sqrt{2}}+\frac{\exp(\frac{-r_{kj}}{\sqrt{2}})\sin\frac{r_{kj}}{\sqrt{2}}}{r_{kj}+\delta_{jk}}\right)_{k,j=1}^m,
  \end{gather*}
where  $r_{kj}=|x_k-x_j|$.
  \end{remark}
\section{Two-dimensional  Schr\"{o}dinger  operator  with  point  interactions}
%
%
%
 \subsection{Boundary  triplet  and Weyl  function}

  Let $H_0^{(1)}(\cdot)$   denote  the Hankel function of  the first  kind  and zero-order.
   It is  known that   the  function $H^{(1)}_0(z)$  has  the following  asymptotic   expansion  at $0$  (see formulas    $(9.01)$ in \cite[chapter 2,]{Olv78} and
 (5.03), (5.07)   in  \cite[chapter 7]{Olv78})
\begin{equation}\label{hank}
H^{(1)}_0(z)=1+\tfrac{2i}{\pi}(
\ln(\tfrac{z}2)-\psi(1))+o(z),\quad |z|\rightarrow 0,\quad \psi(1)=\frac{\Gamma'(1)}{\Gamma(1)}.
\end{equation}

%
%
%
\begin{proposition}\label{pr1a}
Let  $H$ be the minimal Schr\"{o}dinger  operator defined by   \eqref{min}  and  $\xi_0,\xi_1$  be  defined   as  in the previous  case.
Then  the following assertions hold.
\newline
$(i)$ The  operator $H$  is  closed  and symmetric. The deficiency  indices  of    $H$ are $n_\pm(H)=m$. The  defect subspace
$\mathfrak{N}_z := \mathfrak{N}_z(H)$  is
    \begin{equation}\label{4.18'}
\mathfrak{N}_z =\{\sum\limits_{j=1}^mc_j \frac i{4}H_0^{(1)}(\sqrt{z}r_j)\,:\ c_j\in \mathbb{C},\ j\in \{1,\ldots,m\} \}, \quad z\in \C\setminus [0,\infty).
         \end{equation}
 \newline
$(ii)$ The domain  of   $H^*$ is  defined  by
\begin{gather}\label{eq3a}
\dom(H^*) =\left\{ f= \sum\limits^m_{j=1}\bigl(\xi_{0j}\,e^{-r_j}
\ln(r_j)+\xi_{1j}\,e^{-r_j}\bigr) + f_H\ :\
\xi_{0},\xi_{1}\in\bC^m,\quad f_H\in\dom(H)\right\},\\
H^*f=-\sum\limits^m_{j=1}\bigl(\xi_{0j}\frac{e^{-r_j}}{r_j}(r_j\ln(r_j)-\ln(r_j)-2)+\xi_{1j}\frac{e^{-r_j}}{r_j}(1-r_j)\bigr)-\Delta f_H.\label{eq3a'}
\end{gather}
 $(iii)$
 A totality    $\Pi
=\{\kH,\Gamma_0,\Gamma_1\}$,  where
\begin{gather}  \label{T2}
\kH=\bC^m,\quad \Gamma_0 f:=\{\Gamma_{0j}f\}_{j=1}^m=-
2\pi\,\{\lim_{x\rightarrow x_j}\frac
{f(x)}{ \ \ln|x-x_j|}\}^m_{j=1}=2\pi\{\xi_{0j}\}_{j=1}^m,\\
\Gamma_1 f:=\{\Gamma_{1j}f\}_{j=1}^m =\{\lim_{x\rightarrow
x_j}\left(f(x)- \ln|x-x_j|\xi_{0j}\right)\}^m_{j=1},\quad f \in
\dom(H^*)\label{T2'},
\end{gather}
 forms a boundary  triplet  for $H^*$.
\newline
$(iv)$ The  operator $H_0=H^*\upharpoonright \ker(\Gamma_0) (=
H_0^*)$ coincides with the free Hamiltonian,
\[
H_0 = -\Delta, \qquad \dom(H_0) = \dom(-\Delta) = W^{2,2}(\bR^2).
\]
\end{proposition}
\begin{proof}
$(i)$  First  two   statements  are   known (see, for
instance, \cite[chapter II.4]{AGHH88}).  Formula \eqref{4.18'} is
also known.  However we present the proof for   the sake of
completeness.
          %
          %
%
%
The inclusion  $f_j=H^{(1)}_0(\sqrt{z}r_j) \in\mathfrak{N}_z$,
$z\in \C\setminus [0,\infty)$ is amount to saying that
\begin{equation}\label{domH1'}
(H^{(1)}_0(\sqrt{z}r_j), (H-\overline{z})\varphi)=0,\quad
 \varphi\in\dom(H),\quad  j\in\{1,..,m\}.
\end{equation}
Since  $\tfrac i{4}H^{(1)}_0(\sqrt{z}|x-x'|)$  is  the kernel of
the  free Hamiltonian  resolvent   $R_z(H_0)$  (see \cite[chapter
I.5]{AGHH88}),  we  get
\[(\tfrac i{4}H^{(1)}_0(\sqrt{z}|x-x'|), (H_0-\overline{z})\overline{\psi})=R_z(H_0)((H_0-z)\psi)=\psi(x'),\quad\psi\in\dom(H_0).
\]
Hence     for  $\varphi\in\dom(H)$
\[
(H^{(1)}_0(\sqrt{z}r_j), (H-\overline{z})\varphi)=(H^{(1)}_0(\sqrt{z}r_j), (H_0-\overline{z})\varphi)=-4i\overline{\varphi}(x_j)=0,
\]
which  proves \eqref{domH1'}. Thus,  $f_j\in
\ran(H-\overline{z})^\bot=\mathfrak{N}_z,\,\,\,j\in\{1,..,m\}$.

Clearly,   the  functions    $\{H^{(1)}_0(\sqrt{z}r_j)\}_{j=1}^m$,
are   linearly   independent.  Indeed, otherwise
  we have the  equality
\begin{equation}\label{lin}
\sum\limits_{j=1}^m c_jH^{(1)}_0(\sqrt{z}r_j)=0,\quad
\text{with}\quad c_j\in \mathbb{C},\quad
\sum\limits_{j=1}^m|c_j|\neq0.
    \end{equation}
Let  for  definiteness $c_1\neq 0$.  Multiplying    \eqref{lin} by
$\frac1{\ln (r_1)}$,  then passing to  the limit  as  $x$  tends
to $x_1$ and taking  the  asymptotic formula \eqref{hank} into
account, we get $c_1=0$.  
The contradiction proves  \eqref{4.18'}.
%
%
%
%
%
%

$(ii)$ It is known (see \cite{AGHH88,AK1}) that
\begin{displaymath}
\dom (H^*)=\bigl\{f\in L^2(\bR^2)\cap W^{2,2}_{\loc}(\bR^2\backslash X): \Delta f\in L^2(\bR^2)\bigr\}.
\end{displaymath}
%
%
  Obviously,  the functions $f_j= e^{-r_j} \ln(r_j)$
 and $g_j=e^{-r_j}$,\,\, $j\in\{1,..,m\}$ belong to $\dom(H^*)$. Their  linear  independency might  be  derived as  in  3D  case.
 Since
$\dim(\dom(H^*)/\dom(H)) =2m$,  the  domain  $\dom(H^*)$ takes the form
\eqref{eq3a}.

$(iii)$ Let  $f,g\in\dom(H^*)$. By assertion  $(i)$, we  have
 \begin{gather*}
  f= \sum\limits^m_{k=1}f_k + f_H, \,\,\, f_k =\xi_{0k}\,e^{-r_k} \ln(r_k)+\xi_{1k}\,e^{-r_k}, \quad
\mbox{and}\\ g = \sum\limits^m_{k=1}g_k + g_H, \,\,\, g_k =
\eta_{0k}\,e^{-r_k} \ln(r_k)+\eta_{1k}\,e^{-r_k},
 \end{gather*}
 where $f_H,
g_H\in\dom(H)$, and  $\xi_{0k},\ \xi_{1k},\ \eta_{0k},\
\eta_{1k}\in\bC,\ k\in\{1,..,m\}$.

Applying $\Gamma_0, \Gamma_1$ to $f$ and $g$,  we obtain
 \begin{multline}\label{tripl1}
\Gamma_0f=2\pi\{\xi_{0j}\}_{j=1}^m,\quad\Gamma_1f=\left\{\sum\limits_{k\neq
j}\xi_{0k}e^{-|x_j-x_k|}\ln|x_j-x_k|+\sum\limits_{k=1}^m\xi_{1k}e^{-|x_j-x_k|}\right\}_{j=1}^m,\\
\Gamma_0g=2\pi\{\eta_{0j}\}_{j=1}^m,\quad\Gamma_1g=\left\{\sum\limits_{k\neq
j}\eta_{0k}e^{-|x_j-x_k|}\ln|x_j-x_k|+\sum\limits_{k=1}^m\eta_{1k}e^{-|x_j-x_k|}\right\}_{j=1}^m.\qquad\qquad
\end{multline}
 Left-hand side of the Green  identity \eqref{GI} takes  the
form
\begin{multline*}
(H^*f,g) - (f,H^*g) =
\sum\limits^m_{k,j=1}\biggl((\xi_{0j}H^*(e^{-r_j}\ln(r_j)),\eta_{1k}\,e^{-r_k})
-
(\xi_{0j}\,e^{-r_j}\ln(r_j),\eta_{1k}H^*(e^{-r_k}))\\+(\xi_{1j}H^*(e^{-r_j}),\eta_{0k}\,e^{-r_k}\ln(r_k))
- (\xi_{1j}\,e^{-r_j},\eta_{0k}H^*(e^{-r_k}\ln(r_k)))\biggr).
\end{multline*}
Applying  the  second Green formula  yields
\begin{multline}\label{green''}
(H^*(e^{-r_j}\ln(r_j)),e^{-r_k})-
(e^{-r_j}\ln(r_j),H^*(e^{-r_k}))=\\\lim\limits_{r\rightarrow\infty}\int\limits_{B_r(x_j)\backslash
B_{\tfrac1{r}}(x_j)}\left(-\Delta(e^{-r_j}\ln(r_j))e^{-r_k}+e^{-r_j}\ln(r_j)\Delta(e^{-r_k})\right)\mathrm{dx}=-2\pi e^{-|x_k-x_j|}.
\end{multline}
By  \eqref{tripl1} and \eqref{green''}, we obtain
\begin{multline*}
(H^*f,g)
-(f,H^*g)=2\pi\sum\limits^m_{k,j=1}\biggl(-\xi_{0j}\overline{\eta}_{1k}e^{-|x_j-x_k|}+
\xi_{1j}\overline{\eta}_{0k}e^{-|x_j-x_k|}\biggr)\\=
\sum\limits_{j=1}^m(\Gamma_{1j}f,\Gamma_{0j}g)-(\Gamma_{0j}f,\Gamma_{1j}g)=(\Gamma_1f,\Gamma_0g)-(\Gamma_0f,\Gamma_1g).
\end{multline*}
 Thus, the  Green identity is verified. From   \eqref{eq3a}  it follows that the mapping $\Gamma=(\Gamma_0,\Gamma_1)^\top$
is  surjective.   Namely, let $(h_0,h_1)^\top\in\cH\oplus\cH$,
 where $h_0=\{h_{0j}\}_{j=1}^m,h_1=\{h_{1j}\}_{j=1}^m$
 are  vectors from $\mathbb{C}^m$.  If $f\in\dom(H^*)$,
  then, by \eqref{eq3},  $f=f_H+\sum\limits^m_{j=1}\bigl(\xi_{0j}e^{-r_j}\ln(r_j)+\xi_{1j}e^{-r_j}\bigr)$.
  Let us  put
\begin{equation}\label{hat2}
E_0:=\left(e^{-|x_k-x_j|}\ln(|x_k-x_j|+\delta_{kj})\right)_{j,k=1}^m,\quad
E_1:=\left(e^{-|x_k-x_j|}\right)_{k,j=1}^m.
\end{equation}
As  above, invertibility of the matrix $E_1$ might
be derived  from the  fact that the function $e^{-|x|}$ is strictly
positive definite on $\mathbb{R}^2$ (see Example \ref{rempoz}).
Therefore,    setting  $\xi_0=\tfrac1{2\pi}h_0$ and
$\xi_1=E_1^{-1}(h_1-\tfrac1{2\pi}E_0h_0)$,   we get  $\Gamma_0f=h_0$ and
$\Gamma_1f=h_1$.
Thereby, $(iii)$ is  proved.
      \end{proof}
Analogously to  the  previous  case, the  following  parametrization  of the  self-adjoint   extensions is  valid.
   \begin{proposition}
Let   $H$  be the  minimal Schr\"odinger  operator defined by  \eqref{min},
$\Pi =\{\kH,\Gamma_0,\Gamma_1\}$  be  a boundary triplet for $H^*$
defined by  \eqref{T2}-\eqref{T2'} and  let the matrices $E_0, E_1$ be defined
by \eqref{hat2}.
Then the  set   of   self-adjoint  extensions $\widetilde H\in \Ext_H$  is  
parameterized   as
     \begin{equation}\label{s-adjconcr2}
\widetilde H = H_\Theta=H^*\upharpoonright\{
f=\sum\limits^m_{j=1}\bigl(\xi_{0j}e^{-r_j}\ln(r_j) +
\xi_{1j}e^{-r_j} \bigr) +
f_H\in\dom(H^*):\,\{\Gamma_0f,\Gamma_1f\}\in\Theta\},
  \end{equation}
 where $\Theta$ runs  through the   set of all self-adjoint  linear  relations in $\kH$. Namely, 
 $\Theta_{\infty}$ and
 $\Theta_{\op}$ are defined by
%
%
    \begin{gather*}\label{op}
\Theta_{\infty} = \{0, \kH_\infty\} = \{\{0, E_1\xi''_1\}:\,\xi''_1\perp E_1 \xi_0,\  
 \xi_0\in \kH_{\op} \}, \\  
\Theta_{\op} = \{\{2\pi \xi_0, E_0\xi_0 + E_1\xi'_1\}:\, \xi_0\in
\kH_{\op},\,\xi'_1 = E_1^{-1}(2\pi B\xi_0 - E_0\xi_0)\}, \label{mul}
  \end{gather*}
where $B=B^*\in [\kH_{\op}]$. In particular, $\widetilde H = H_\Theta$
is disjoint with $H_0$ if and only if $\dom(\Theta)=\mathbb{C}^m$.
In this case $\Theta=\Theta_{\op}= grB.$
\end{proposition}

\begin{remark}
The $m$-parametric family $H^{(2)}_{X,\alpha}$
investigated  in \cite[Theorem 1.1.1]{AGHH88} is parameterized by
the   diagonal matrices  $\Theta=\Theta_{\op}=B_\alpha = \diag(\alpha_1,..,\alpha_m), \,\,\,\alpha_j\in  \mathbb{R}.$
\begin{equation*}
H^{(2)}_{X,\alpha}=H^*\upharpoonright\{ f=\sum\limits^m_{j=1}\xi_{0j}e^{-r_j}\ln(r_j)+\sum\limits^m_{k,j=1}b_{jk}(\alpha)\xi_{0k}e^{-r_j}+f_H\},
\end{equation*}
where $\widetilde{B}=(b_{jk}(\alpha))_{j,k=1}^m=E_1^{-1}(2\pi
B_\alpha-E_0)$.
\end{remark}

   \begin{proposition}\label{pr2a}
Let  $H$ be the minimal Schr\"{o}dinger operator  and let
$\Pi=\{\kH,\Gamma_0,\Gamma_1\}$  be the boundary triplet  for
$H^*$ defined by \eqref{T2}-\eqref{T2'}. Then

$(i)$
the Weyl function $M(\cdot)$ corresponding to the boundary
triplet $\Pi$ is
\begin{equation}\label{W2}
M(z)=\left(\tfrac1{2\pi}(\psi(1)-
\ln(\tfrac{\sqrt{z}}{2i}))\delta_{jk}+\widetilde{G}_{\sqrt{z}}(x_j-x_{k})\right)_{j,k=1}^m,
\quad z\in\mathbb{C}_+,
 \end{equation}
  where $\psi(1)=\frac{\Gamma'(1)}{\Gamma(1)}$, \quad
      $\widetilde{G}_{\sqrt{z}}(x)=\left\{%
  \begin{array}{ll}
    i/4 H_0^{(1)}(\sqrt{z}|x|), & \hbox{x$\neq$0;} \\
    0, & \hbox{x=0.} \\
\end{array}%
\right.$

$(ii)$  the corresponding $\gamma(\cdot)$-field is given by
\[\gamma(z)\{a_j\}^m_{j=1} =
\sum\limits^m_{j=1}a_j \frac i{4} H_0^{(1)}(\sqrt{z}|x-x_j|).\]
 \end{proposition}
    \begin{proof}
Combining  \eqref{4.18'}  with  \eqref{T2}-\eqref{T2'}  and taking
into  account    expansion \eqref{hank},  we arrive  at  $(i)$
and $(ii)$.
    \end{proof}
 \begin{remark}
Note  that  the Weyl  function  in the  form \eqref{W2}  appears   in  \cite[chapter II.4, Theorem 4.1]{AGHH88}.
In the case  $m=1$,  the Weyl  function   was  also   computed  in \cite[section 10.3]{Ash10}.
   \end{remark}

\subsection{Spectrum  of the  self-adjoint  extensions of  the  minimal Schr\"{o}dinger  operator}

As  above,  the  following   lemma  holds.
\begin{lemma}
Let  $H^*$  be  defined  by \eqref{eq3a}-\eqref{eq3a'}. Then $0\notin\sigma_p(H^*).$
\end{lemma}
 \begin{proof}
 The  proof  repeats    the  one in the 3D  case. It only  should be  noted that
 the  functions  $\ln(r_j),\,\, j\in\{1,..,m\},$ are  harmonic in $\mathbb{R}^2\setminus X$ and
 $f_0\in\dom(H^*)$ admits the  representation
\begin{equation*}\label{harm'}
f_0 =\sum\limits_{j=1}^m \xi_{0j}\ln(r_j) +
g,\quad\text{with}\quad g :=
\sum\limits_{j=1}^m\bigl(\xi_{0j}(e^{-r_j}-1)\ln(r_j)+\xi_{1j}e^{-r_j}\bigr)+f_H\in
W^{2,2}_{\loc}(\mathbb{R}^2),
\end{equation*}
  where   $(e^{-r_j}-1)\ln(r_j)\in W^{2,2}_{\loc}(\mathbb{R}^2)$ for any $j\in\{1,..,m\}$.
    \end{proof}
Spectrum  of  the self-adjoint  extensions of    $H$ is described
in the following  theorem.
\begin{theorem}\label{spec2}
Let  $H$ be the  minimal Schr\"odinger  operator defined by  \eqref{min}, let  $\Pi$ be
the boundary triplet for  $H^*$ defined by \eqref{T2}-\eqref{T2'},
and  $M(\cdot)$ be  the corresponding Weyl function  defined  by
\eqref{W2}. Assume also that $H_\Theta=H_\Theta^*\in \Ext_H$ is
defined  by \eqref{s-adjconcr2}.
 Then  the following assertions hold.

$(i)$ The extension $H_\Theta$ has  purely absolutely continuous
non-negative spectrum   of infinite multiplicity.

 $(ii)$ Point spectrum of
the self-adjoint  extension  $H_\Theta$  consists of  at most $m$
negative eigenvalues (counting multiplicities). Moreover,
$z\in\sigma_p(H_\Theta)\cap\bR_-$
 if and only  if $0\in\sigma_p(\Theta-M(z))$, i.e.,
 \begin{equation*}
 z\in\sigma_p(H_\Theta)\cap\bR_- \Longleftrightarrow 0\in\sigma_p\bigl(\Theta-\left(\tfrac1{2\pi}(\psi(1)-
\ln(\tfrac{\sqrt{z}}{2i}))\delta_{jk}+\widetilde{G}_{\sqrt{z}}(x_j-x_{k})\right)_{j,k=1}^m\bigr).
\end{equation*}
The corresponding   eigenfunctions $\psi_z$ have  the  form
\begin{equation*}\label{eigen'}
\psi_z=\sum\limits_{j=1}^ma_j\tfrac{i}4H^{(1)}_0(\sqrt{z}r_j),\qquad \text{where}\quad(a_1,.., a_m)^\top\in \ker(\Theta-M(z)).
\end{equation*}
\end{theorem}
\begin{proof}
The  proof  is similar to that of Theorem
\ref{spec3}.
It  follows  from \eqref{W2}  that
\begin{equation}\label{W2lim}
M(x+i0)=\left(\tfrac1{2\pi}(\psi(1)-
\ln(\tfrac{\sqrt{x}}{2i}))\delta_{jk}+\widetilde{G}_{\sqrt{x}}(x_j-x_{k})\right)_{j,k=1}^m,\quad x\in\mathbb{R}.
\end{equation}
Note    that,   according to \cite[chapter 7,\S 4]{Olv78},
 \begin{equation*}
 \imm(\widetilde{G}_{\sqrt{x}}(x_j-x_{k}))=\tfrac1{4}J_0(\sqrt{x}|x_j-x_{k}|),\quad x\geq 0,
 \end{equation*}
 where  $J_0(\cdot)$  is the  Bessel function.
 Combining  this  fact with \eqref{W2lim},  we get

\begin{equation*}
\imm(M(x+i0))=\left(\tfrac1{4}J_0(\sqrt{x}|x_j-x_{k}|)\right)_{j,k=1}^m,\quad
x\geq 0.
\end{equation*}
 In  a view  of  \eqref{kernel}, $J_0(sx)\in\Phi_2,\,s>0$ (see Example \ref{rempoz}), and therefore
 the  matrix $\imm(M(x+i0))$  is   positive   definite for
 $x\in\mathbb{R}_+$.
   \end{proof}
%
%
    \begin{remark}
The   description  of   absolutely   continuous    and   point   spectrum  in  the  particular  case  of  the  family  $H_{X,\alpha}^{(2)}$  was  obtained  in  \cite[Theorem 1.1.4]{AGHH88}.
    \end{remark}

\end{document}